\documentclass[a4paper,11pt]{article}
\title{ Universal $L^s$-rate-optimality of $L^r$-optimal quantizers by dilatation and contraction  }
 \author{\textit{Abass  Sagna}}
\usepackage[T1]{fontenc}
\usepackage{amsfonts}
\usepackage[english]{babel}
\usepackage{a4wide,times}

\usepackage[latin1]{inputenc}
\usepackage{amssymb}
\usepackage{graphicx}
\usepackage{epsfig}
\usepackage{amsbsy}
\usepackage{verbatim}
\usepackage{color}
\usepackage{mathrsfs}
\usepackage{makeidx}
\usepackage{amsmath}
\usepackage{amsthm}
\date{}
\newtheorem{thm}{Theorem}[section]
\newtheorem{prop}{Proposition}[section]
\newtheorem{cor}{Corollary}[section]
\newtheorem{crit}{Criterion}[section]
\newtheorem{conj}{Conjecture}
\newtheorem{lem}{Lemma}[section]
\newtheorem{rem}{Remark}[section]
\newtheorem{defi}{Definition}[section]
 \newtheorem*{pre*}{Preuve} 

\newenvironment{maliste}
{ \begin{list}
              {$\bullet $}
              {\setlength{\labelwidth}{30pt}
              \setlength{\leftmargin}{35pt}
              \setlength{\itemsep}{\parsep}}}
{ \end{list}}
\numberwithin{equation}{section}

\begin{document}
\maketitle{}

{\small \textit{Laboratoire de probabilit\'es et mod\`eles al\'eatoires, UMR7599, Universit\'e Pierre et Marie Curie,}}

 {\small \textit{Case 188, 4 place Jussieu, 75252 Cedex 05, Paris, France. E-mail address: sagna@ccr.jussieu.fr.}}
\vskip 1.5cm
\begin{abstract}
Let $ r, s>0  $. For a given probability measure $P$ on $\mathbb{R}^d$, let  $(\alpha_n)_{n \geq 1}$ be a sequence of  (asymptotically) $L^r(P)$- optimal quantizers. For all  $\mu \in \mathbb{R}^d $ and for every  $\theta >0$, one defines the sequence $(\alpha_n^{\theta, \mu})_{n \geq 1}$ by : $\forall n \geq 1, \ \alpha_n^{\theta, \mu} = \mu + \theta(\alpha_n - \mu) = \{ \mu + \theta(a- \mu), \ a \in \alpha_n \} $. In this paper, we are interested in the asymptotics of the $L^s$-quantization error induced by the sequence $(\alpha_n^{\theta, \mu})_{n \geq 1}$. We show that for a wide family of distributions, the sequence $(\alpha_n^{\theta, \mu})_{n \geq 1}$ is $L^s$-rate-optimal. For  the Gaussian and the  exponential distributions, one shows  how to choose the parameter $\theta$ such that $(\alpha_n^{\theta, \mu})_{n \geq 1}$ satisfies the empirical measure theorem and probably be asymptotically $L^s$-optimal.
\end{abstract}
\section{ Introduction}

Let $(\Omega,\mathcal{A},\mathbb{P})$ be a  probability space and  let   $X : (\Omega,\mathcal{A},\mathbb{P}) \longrightarrow  \mathbb{R}^{d} $ be a random variable with  distribution $\mathbb{P}_X=P$. Let $\alpha\subset \mathbb{R}^d$ be a subset (a codebook) of size $n$. A  Borel partition $C_a(\alpha)_{a \in \alpha}$ of $\mathbb{R}^d$   satisfying $$ C_a(\alpha) \subset \{ x \in \mathbb{R}^d : \vert x-a \vert = \min_{b \in \alpha_n}\vert x-b \vert \},$$
where $ \vert \cdot \vert $ denotes a  norm on $\mathbb{R}^d$ is called a Voronoi partition of  $\mathbb{R}^d$ (with respect to $\alpha$ and $ \vert \cdot \vert $). \\
 The random variable  $\widehat{X}^{\alpha}$ taking values in the codebook  $\alpha$ defined by 
$$ \widehat{X}^{\alpha} = \sum_{a \in \alpha} a \mbox{\bf{1}}_{\{X \in C_a(\alpha)\}} . $$ is called a Voronoi quantization of $X$.  In other words, it is the  nearest neighbour projection of $X$ onto the codebook  (also called grid)  $\alpha$.

The $n$-$L^r(P)$-optimal quantization  problem for $P$ (or  $X$) consists in the study of the best approximation of  $X$  by a Borel function taking  at most  $n$ values. For $X \in L^r(\mathbb{P})$ this leads to the following  optimization problem:
$$ e_{n,r}(X) = \inf{\{ \Vert X - \widehat{X}^{\alpha} \Vert_r, \alpha \subset \mathbb{R}^d,  \textrm{ card}(\alpha) \leq n \}}  $$
 with  $$\Vert X - \widehat{X}^{\alpha} \Vert_r^r = \mathbb{E} \big( d(X,\alpha) \big)^r = \int_{\mathbb{R}^d}d(x,\alpha)^rdP(x). $$ 
Then we can write
\begin{equation} \label{er.quant}
 e_{n,r}(X) = e_{n,r}(P) = \inf_{ \substack{\alpha \subset \mathbb{R}^d \\ \textrm{card}(\alpha) \leq n}} \left(\int_{\mathbb{R}^d} d(x,\alpha)^r dP(x) \right)^{1/r}.
\end{equation}
We remind in what follows some definitions and results  that  will be used throughout the paper. 

\begin{itemize}
\item  For all $n \geq 1$, the infimum in $(\ref{er.quant})$ is reached at  one (at least) grid $\alpha^{\star}$;  $\alpha^{\star} $ is then called a  $L^r$-optimal $n$-quantizer. In addition, if  $\textrm{ card(supp}(P)) \geq n $ then card$(\alpha^{\star}) = n$ (see $\cite{GL}$ or $\cite{P}$). Moreover the quantization error, $e_{n,r}(X)$, decreases to zero as $n$ goes to infinity and the so-called Zador's Theorem mentionned below gives its convergence rate provided a moment assumption on $X$.

\item   Let $X \sim P$ and let $P=P_a + P_s$ be the Lebesgue decomposition of $P$ with respect to the Lebesgue measure $\lambda_d$, where $P_a$ denotes the absolutely  continuous  part and $P_s$ the singular part of $P$.
 
\textbf{Zador Theorem }(see $\cite{GL}$) :  Suppose  $\mathbb{E}\vert X \vert^{r+ \eta} < + \infty \textrm{ for some } \eta >0 $. Then $$  \lim_{n \rightarrow +\infty} n^{r/d} (e_{n,r}(P))^r = Q_r(P).$$
with $$ Q_r(P) = J_{r,d} \left( \int_{\mathbb{R}^d} f^{\frac{d}{d+r}} d\lambda_d \right)^{\frac{d+r}{d}} = J_{r,d} \ \Vert f \Vert_{\frac{d}{d+r}} \ \in [0,+\infty), $$
 $$ J_{r,d} = \inf_{n \geq 1} e_{n,r}^r(U([0,1]^d)) \in (0,+\infty), $$
where $ U([0,1]^d) $ denotes the uniform distribution on the set  $[0,1]^d$ and  $f = \frac{dP_{a}}{d\lambda_d}$. Note that the moment assumption : $\mathbb{E}\vert X \vert^{r+ \eta} < + \infty $ ensure that  $\Vert f \Vert_{\frac{d}{d+r}} $ is finite. Furthermore,  $ \ Q_r(P) >0 \ $ if and only if  $P_a$  does not vanish. 
\item A sequence of $n$-quantizers $(\alpha_n)_{n \geq 1}$ is 

\begin{enumerate}
\item[-]  \textbf{$L^r(P)$-rate-optimal}  (or \textbf{rate-optimal} for $X$, $X \sim P$) if
$$ \limsup_{n \rightarrow +\infty}n^{1/d} \int_{\mathbb{R}^d}d(x,\alpha_n)^r dP(x) < +\infty,  $$ 
     
\item[-]  \textbf{asymptotically $L^r(P)$-optimal}  if
$$ \lim_{n \rightarrow +\infty} n^{r/d} \int_{\mathbb{R}^d}d(x,\alpha_n)^r dP(x) = Q_r(P) $$ 

 \item[-]  \textbf{$L^r(P)$-optimal}  if for all $n \geq 1$,
$$e_{n,r}^r(P) =  \int_{\mathbb{R}^d} d(x,\alpha_n)^r dP(x) .$$
\end{enumerate}

\item \textbf{Empirical  measure theorem} (see $\cite{GL}$) : Let $X \sim P$. Suppose $P$ is absolutely continuous with respect to $\lambda_d$  and   $\mathbb{E}\vert X \vert^{r+ \eta} < + \infty $ for some  $ \eta >0 $. Let  $(\alpha_n)_{n \geq 1}$  be an  asymptotically  $L^r(P)$-optimal sequence of quantizers. Then
 
\begin{equation} \label{empthm}
 \frac{1}{n} \sum_{a \in \alpha_n} \delta_a \stackrel{w}{\longrightarrow} P_r 
\end{equation}
 where  $ \stackrel{w}{\longrightarrow} $ denotes the weak convergence and for every Borel set $A$ of $\mathbb{R}^d , \ P_r$ is defined by
 \begin{equation} \label{eqmesemp}
P_r(A) = \frac{1}{C_{f,r}}  \int_A f(x)^{\frac{d}{d+r}} d\lambda_d(x) , \ \textrm{with}  \quad C_{f,r} = \int_{\mathbb{R}^d} f(x)^{\frac{d}{d+r}}  d\lambda_d(x) .
\end{equation}

\item    In $ \cite{GLP} $ is  established the following proposition.  
 
\textbf{Proposition :} \  Let  $X \sim P, \textrm{ with } P_a \not= 0 $,  such  that $\mathbb{E} \vert X  \vert^{r+\eta}  < \infty,  \textrm{ for some } \eta > 0$ .   Let    $(\alpha_n)$ be an  $L^r(P)$-optimal sequence of quantizers,  $ b \in \big(0,1/2 \big) \textrm{ and let  } \psi_b : \mathbb{R}^d \longrightarrow \mathbb{R}_{+} \cup \{+\infty \}  $ be the maximal function defined by
\begin{equation} \label{defoncmax}
\psi_b(x) = \sup_{n \geq 1} \frac{\lambda_d(B(x,bd(x,\alpha_n)))}{P(B(x,bd(x,\alpha_n)))} .
\end{equation}
 Then for every  $x \in \mathbb{R}^d$, 
\begin{equation} \label{propGLP}
 \forall n \geq 1, \quad  n^{1/d} d(x,\alpha_n) \leq C(b) \psi_b(x)^{1/(d+r)} 
\end{equation} 
 where  $C(b)$ denotes  a real constant  not  depending on $n$.

\item[$\bullet$]  The next proposition  is established   in $\cite{FP}$. It is used to compute the $L^r$-optimal quantizers for the exponential distribution.

  \textbf{Proposition} \  Let $r>0$ and let $X$ be an exponentially distributed random variable with scale parameter $\lambda>0$. Then for every  $n \geq 1$, the $L^r$-optimal quantizer $\alpha_n = (\alpha_{n1}, \cdots, \alpha_{nn})$ is unique and given by 
\begin{equation}
\alpha_{nk} = \frac{a_n}{2}  +  \sum_{i=n+1-k}^{n-1} a_i, \quad 1 \leq k \leq n,
\end{equation}   
where  $(a_k)_{k \geq 1}$ is a $\mathbb{R}_{+}$-valued sequence defined by the following implicit recursive equation: 
\begin{equation*}
a_0 := +\infty, \quad \phi(-a_{k+1}) := \phi(a_k), \ k \geq 0 
\end{equation*}
with  $\phi(x) := \int_0^{x/2} \vert u \vert^{r-1} \textrm{sign}(u) e^{-u} du $ ( convention : $ 0^0 = 1$).

Furthermore, the sequence  $(a_k)_{k \geq 1}$ decreases to zero and  for every  $k \geq 1$, $$a_k = \frac{r+1}{k}\big(1+ \frac{c_r}{k} + \mbox{O}(\frac{1}{k^2}) \big) $$
for some real constant $c_r$. 

\end{itemize}
{\sc notations } 
\begin{maliste}
\item Let $\alpha_n$ be a set of  $n$ points of $\mathbb{R}^d$ . For every  $\mu \in \mathbb{R}^d $ and  every $\theta >0 $ we denote  \\ $\alpha_n^{\theta,\mu} = \mu + \theta(\alpha_n - \mu) = \{\mu + \theta(a - \mu),\ a \in \alpha_n \}$. 
\item Let  $f :\mathbb{R}^d \longrightarrow \mathbb{R}^d $ be a Borel function and let $\mu \in \mathbb{R}^d , \theta>0$. One notes by  $f_{\theta,\mu}$ (or  $f_{\theta}$ if $\mu=0$) the function defined by $f_{\theta,\mu}(x) =f(\mu + \theta(x-\mu)), \ x\in \mathbb{R}^d. $
\item  If  $X \sim P$, $P_{\theta,\mu}$ will stand for  the probability measure of the random variable  $\ \frac{X-\mu}{\theta} + \mu, \ \theta>0, \mu \in \mathbb{R}^d $. In other words, it is the distribution image of $P$  by $x \longmapsto \frac{x-\mu}{\theta} + \mu.$ Note that if $P=f \cdot \lambda_d$ then $P_{\theta,\mu} = f_{\theta,\mu} \cdot \lambda_d$.

\item If $A$ is  a matrix $A'$   stands for  its transpose. 
\item  Set $x=(x_1,\cdots,x_d); \ y=(y_1,\cdots,y_d) \in \mathbb{R}^d$; we denote $[x,y]= [x_1,y_1] \times  \cdots \times [x_d,y_d]$.
\end{maliste}
\vskip 0.4cm
\begin{defi}
A sequence of quantizers $(\beta_n)_{n \geq 1}$  is called a \textbf{dilatation} of the sequence $(\alpha_n)_{n \geq 1} $ with \textbf{scaling number}  $\theta$ and \textbf{translating number} $\mu$ if, for every $n \geq 1, \  \beta_n = \alpha_n^{\theta,\mu}$, with  $\theta > 1$. If $\theta < 1 $, one defines likewise  the \textbf{contraction}  of the sequence $(\alpha_n)_{n \geq 1} $ with \textbf{scaling number}  $\theta$ and \textbf{translating number} $\mu$. 
 \end{defi}
\section{Lower  estimate}
Let $r,s >0$. Consider an asymptotically $L^r(P)$-optimal sequence of quantizers $(\alpha_n)_{n \geq 1}$ . For every  $\mu \in \mathbb{R}^d $ and  every $\ \theta >0$,  we  construct the sequence  $(\alpha_n^{\theta,\mu})_{n \geq 1}$ and try to lower bound asymptotically the $L^s$-quantization  error induced by this sequence. This estimation provides a  necessary  condition of rate-optimality for the sequence $(\alpha_n^{\theta,\mu})_{n \geq 1}$. Obviously, in the particular case where $\theta=1$ and $\mu=0$ we get the same result as in $\cite{GLP}$, a paper we will essentially draw on.

\begin{thm} \label{thm1}  Let  $r,s \in (0,+\infty )$, and let $X$ be a random variable taking values in $\mathbb{R}^d$  with  distribution $P$  $\textrm{ such that } \ P_a =  f.\lambda_d \not \equiv 0$. Suppose that $\mathbb{E}\vert X \vert^{r+\eta} < \infty \textrm{ for some } \eta >0$.
 Let $(\alpha_n)_{n \geq 1}$ be an asymptotically $L^r(P)$-optimal sequence of quantizers. Then, for every  $\theta >0$  and  every $\mu \in \mathbb{R}^d $,
\begin{equation} \label{liminf} 
 \liminf_{n \rightarrow + \infty}  n^{s/d} \ \Vert X - \widehat{X}^{\alpha_n^{\theta,\mu}} \Vert_s^s \geq Q_{r,s}^{\textrm{Inf}}(P,\theta),
 \end{equation}
with
$$ Q_{r,s}^{\textrm{Inf}}(P,\theta) = \theta^{s+d}  J_{s,d} \left(  \int_{\mathbb{R}^d} f^{\frac{d}{d+r}} d\lambda_d   \right)^{s/d} \int_{\{f>0\}} f_{\theta, \mu} f^{- \frac{s}{d+r}} d\lambda_d. $$

\end{thm}
\vskip 0.4cm
\begin{proof}[\textbf{Proof}]
 Let $ m \geq 1 $ and 
  \[ f_m^{\theta,\mu} = \sum_{k,l=0}^{m2^m-1} \frac{l}{2^m} \mbox{\bf{1}}_{E_k^m \cap G_l^m } ; \]
 with $$ E_k^m = \left\{ \frac{k}{2^m} \leq f < \frac{k+1}{2^m} \right\} \cap B(0,m) \textrm{ and } G_l^m = \left\{ \frac{l}{2^m} \leq f_{\theta,\mu} < \frac{l+1}{2^m} \right\} \cap B(0,m) . $$
 The sequence $ (f_m^{\theta,\mu})_{m \geq 1} $ is non-decreasing and  $$ \lim_{m \rightarrow + \infty} f_m^{\theta,\mu} = f_{\theta,\mu}   \qquad \lambda_d \textrm{ p.p}.  $$
Let $$ I_m = \{ (k,l) \in \{ 0, \cdots,m2^m-1 \}^2 : \lambda_d(E_k^m)>0; \lambda_d(G_l^m)>0  \}. $$
For every  $ (k,l) \in I_m $ there exists compact sets  $K_k^m  \textrm{ and } L_l^m $  such that :
$$ K_k^m \subset E_k^m \textrm{ , } L_l^m \subset G_l^m  \textrm{ , } \lambda_d (E_k^m \backslash K_k^m ) \leq \frac{1}{m^4 2^{2m+1}} \textrm{ and }  \lambda_d (G_l^m \backslash  L_l^m ) \leq \frac{1}{m^4 2^{2m+1}} . $$
Then
\begin{eqnarray*}
 (E_k^m \cap G_l^m )\backslash (K_k^m \cap L_l^m) 
& = &  E_k^m \cap G_l^m \cap ((K_k^m)^{c} \cup (L_l^m)^{c}) \\
& \subset &  (E_k^m \backslash K_k^m) \cup (G_l^m \backslash L_l^m).
\end{eqnarray*}
Consequently,
\begin{eqnarray*}
 \lambda_d(E_k^m \cap G_l^m \backslash K_k^m \cap L_l^m) 
& \leq & \lambda_d(E_k^m \backslash K_k^m) + \lambda_d(G_l^m \backslash L_l^m) \\ & \leq &
\frac{1}{m^4 2^{2m+1}} + \frac{1}{m^4 2^{2m+1}} \\ & = &
\frac{1}{m^4 2^{2m}}.
\end{eqnarray*}
For every $m \geq 1$ and every $(k,l) \in I_m$, set  

 \[ A_{k,l}^m := K_k^m \cap L_l^m, \]

\[ \tilde{f}_m^{\theta,\mu} := \sum_{k,l=0}^{m2^m-1} \frac{l}{2^m} \mbox{\bf{1}}_{A_{k,l}^m},  \] 
and
$$ \tilde{f}_m := \sum_{k,l=0}^{m2^m-1} \frac{k}{2^m} \mbox{\bf{1}}_{A_{k,l}^m}. $$
We get $$ \{ f_m^{\theta,\mu} \not= \tilde{f}_m^{\theta,\mu} \} \subset \bigcup_{k,l \in \{0,\cdots,m2^m-1 \}} \left((E_k^m \cap G_l^m)  \backslash A_{k,l}^m \right).  $$
Therefore, for every  $m \geq 1,$  $$ \lambda_d(\{f_m^{\theta,\mu} \not= \tilde{f}_m^{\theta,\mu} \}) \leq \sum_{k,l=0}^{m2^m-1} \frac{1}{m^4 2^{2m}} = \frac{1}{m^2}$$
and finally   $$ \sum_{m \geq 1} \mbox{\bf{1}}_{\{ f_m^{\theta,\mu} \not= \tilde{f}_m^{\theta,\mu} \}} < \infty  \qquad \lambda_d \textrm{ p.p}. $$
As a consequently  $ \lambda_d(dx) \textrm{-p.p}  \textrm{,\ }  f_m^{\theta,\mu}(x) =  \tilde{f}_m^{\theta,\mu}(x) $ for large enough $m$. Then  $ \tilde{f}_m^{\theta,\mu} \stackrel{\lambda_d  \  p.p.}{\longrightarrow} f_{\theta,\mu}  \textrm{ when } \\ m \rightarrow +\infty.$  Since in addition $A_{k,l}^m \subset E_k^m \cap G_l^m $  we obtain
\begin{equation} \label{eq}
 \tilde{f}_m^{\theta,\mu} \leq f_m^{\theta,\mu} \leq f_{\theta,\mu} .
\end{equation}
Moreover, for every $n \geq 1$, 
\begin{eqnarray*}
 n^{s/d} \ \Vert X - \widehat{X}^{\alpha_n^{\theta,\mu}}  \Vert_s^{s} 
& = & n^{s/d} \int_{\mathbb{R}^d} d(z, \mu +\theta (\alpha_n +\mu))^s f(z) \lambda_d(dz) \\ & \geq &
n^{s/d} \int_{\mathbb{R}^d} \min_{a \in \alpha_n} \vert z-(\mu+\theta(a-\mu)) \vert^s f(z) \lambda_d(dz) \\ & \geq &
\theta^s n^{s/d} \int_{\mathbb{R}^d} \min_{a \in \alpha_n}\vert (z-\mu)/ \theta + \mu - a \vert^s f(z) \lambda_d(dz).
\end{eqnarray*}
Making the change of variable  $x :=(z-\mu)/ \theta + \mu$ yields:
\begin{eqnarray} \label{suradd}
n^{s/d} \ \Vert X - \widehat{X}^{\alpha_n^{\theta,\mu}}  \Vert_s^{s} 
& \geq & \theta^{s+d} n^{s/d} \int_{\mathbb{R}^d} d(x,\alpha_n)^s f_{\theta,\mu}(x) \lambda_d(dx) \nonumber \\ & \geq &
\theta^{s+d} n^{s/d} \int_{\mathbb{R}^d} d(x,\alpha_n)^s \tilde{f}_m^{\theta,\mu} \lambda_d(dx) \qquad \textrm{( by  (\ref{eq}) )}  \nonumber  \\
& = & \theta^{s+d} n^{s/d} \sum_{k,l=0}^{m2^m-1} \frac{l}{2^m} \int_{A_{k,l}^m} d(x,\alpha_n)^s \lambda_d(dx).
\end{eqnarray}
 Let  $m \geq 1 $ and $ (k,l) \in I_m$. Define the closed sets $\tilde{A}_{k,l}^m$ by  $\tilde{A}_{k,l}^m=\emptyset$ if $\lambda_d(\tilde{A}_{k,l}^m) = 0$ and otherwise by  $$ \tilde{A}_{k,l}^m = \{ x \in \mathbb{R}^d : d(x,A_{k,l}^m) \leq \varepsilon_m \} $$
where $\varepsilon_m \in (0,1]$ is chosen 
so that $$\int_{\tilde{A}_{k,l}^m} f^{\frac{d}{d+r}}d\lambda_d \leq \big(1+ 1/m \big) \int_{A_{k,l}^m} f^{\frac{d}{d+r}}d\lambda_d. $$
Since $   \tilde{A}_{k,l}^m  \textrm{ is compact }( \tilde{A}_{k,l}^m   \subset B(0,m+1)) \textrm{ } \forall (k,l)  $, there exists  ( ref. \cite{DGLP}, \textit{Lemma 4.3}) a finite \\ " firewall" set  $\beta_{k,l}^m $ such that 
$$\forall n \geq 1, \quad \forall x \in \tilde{A}_{k,l}^m, \quad d(x,\alpha_n \cup \beta_{k,l}^m)=d(x,(\alpha_n \cup \beta_{k,l}^m)\cap \tilde{A}_{k,l}^m). $$
The last inequality is in particular  satisfied for all  $ x \in A_{k,l}^m  $ since  $  A_{k,l}^m \subset \tilde{A}_{k,l}^m$. \\
Now set $ \beta^{m} = \bigcup_{k,l} \beta_{k,l}^m \ \textrm{ and } \ n_{k,l}^m = \textrm{card}((\alpha_n \cup \beta^m)\cap \tilde{A}_{k,l}^m) .$  The empirical measure theorem (see $(\ref{empthm})$) yields  
$$ \limsup_n    \frac{ \textrm{card}(\alpha_n \cap \tilde{A}_{k,l}^m)}{n} =  \frac{\int_{\alpha_n \cap \tilde{A}_{k,l}^m} f^{\frac{d}{d+r}} d\lambda_d}{\int  f^{\frac{d}{d+r}} d\lambda_d }  \leq  \frac{\int_{\tilde{A}_{k,l}^m} f^{\frac{d}{d+r}} d\lambda_d}{\int  f^{\frac{d}{d+r}} d\lambda_d } . $$
Moreover  $$ \frac{n_{k,l}^m}{n} \sim  \frac{ \textrm{card}(\alpha_n \cap \tilde{A}_{k,l}^m)}{n}  \qquad  \textrm{ when }  \ n \rightarrow + \infty $$
then
\begin{equation}  \label{ineq1GLP} 
 \liminf_{n \rightarrow +\infty}  \frac{n}{n_{k,l}^m}  \geq   \frac{\int f^{\frac{d}{d+r}} d \lambda_d}{\int_{\tilde{A}_{k,l}^m} f^{\frac{d}{d+r}} d\lambda_d}  \geq  \frac{m}{m+1} \frac{\int f^{\frac{d}{d+r}} d \lambda_d}{\int_{A_{k,l}^m} f^{\frac{d}{d+r}} d\lambda_d}. 
 \end{equation}
 In the other hand, 
\begin{eqnarray*}
 \int_{A_{k,l}^m} d(x,\alpha_n)^s  \lambda_d(dx) 
& \geq & \int_{A_{k,l}^m } d(x,(\alpha_n \cup \beta_{k,l}^m) \cap \tilde{A}_{k,l}^m)^s \lambda_d(dx) \\
& = & \lambda_d(A_{k,l}^m) \int d(x,(\alpha_n \cup \beta_{k,l}^m) \cap \tilde{A}_{k,l}^m)^s \mbox{\bf{1}}_{A_{k,l}^m}(x) \frac{\lambda_d(dx)}{\lambda_d(A_{k,l}^m)} \\
& \geq &  \lambda_d(A_{k,l}^m)e_{n_{k,l}^m,s}^s(U(A_{k,l}^m)),
\end{eqnarray*}
where $U(A) = \mbox{\bf{1}}_A / \lambda_d(A)$  denotes the uniform distribution in the Borel set $A$ when $\lambda_d(A) \not= 0.$
Then we can write for every $(k,l) \in I_m$,
$$\liminf_{n\rightarrow +\infty}  n^{s/d}\int_{A_{k,l}^m} d(x,\alpha_n)^s \lambda_d(dx)  \geq  \lambda_d(A_{k,l}^m) \liminf_{n} \left( \frac{n}{n_{k,l}^m}  \right)^{s/d} \liminf_{n} n^{s/d} e_{n,s}^s(U(A_{k,l}^m)) ,  $$
since   $$ \liminf_{n} n^{s/d} e_{n,s}^s(U(A_{k,l}^m))  \geq  J_{s,d} \cdot \lambda_d(A_{k,l}^m) ^{s/d}. $$
Owing to  Equation  $(\ref{ineq1GLP})$, one has
$$ \liminf_{n\rightarrow +\infty}  n^{s/d}\int_{A_{k,l}^m} d(x,\alpha_n)^s \lambda_d(dx)   \geq  \lambda_d(A_{k,l}^m) \left( \frac{m}{m+1} \frac{\int f^{\frac{d}{d+r}} d \lambda_d}{\int_{A_{k,l}^m} f^{\frac{d}{d+r}} d\lambda_d} \right)^{s/d}  J_{s,d} \cdot \lambda_d(A_{k,l}^m) ^{s/d}. $$
However, on the sets  $ A_{k,l}^m,  \ \textrm{ the statement }  \quad   \frac{1}{f} \geq  \left( \frac{k+1}{2^m} \right)^{-1}$  holds   since  $f < \frac{k+1}{2^m}  $ on   $E_k^m$. Hence
 $$ \liminf_{n\rightarrow +\infty}  n^{s/d}\int_{A_{k,l}^m} d(x,\alpha_n)^s \lambda_d(dx)  \geq  J_{s,d}  \left( \frac{m+1}{m} \int f^{\frac{d}{d+r}} \lambda_d(dx) \right)^{s/d} \left( \frac{k+1}{2^m} \right)^{-\frac{d}{d+r} \cdot \frac{s}{d}} \lambda_d(A_{k,l}^m).$$
It follows from Equation  $(\ref{suradd})$ and the super-additivity of the liminf that  for every  $m \geq 1$,
\begin{eqnarray*}
\liminf_{n}   n^{s/d} \ \Vert X - \widehat{X}^{\alpha_n^{\theta,\mu}}  \Vert_s^{s} 
& \geq & \theta^{s+d}J_{s,d}  \left( \frac{m+1}{m} \int f^{\frac{d}{d+r}} \lambda_d(dx) \right)^{s/d}  \sum_{k,l=0}^{m2^m-1} \frac{l}{2^m} \left( \frac{k+1}{2^m} \right)^{-\frac{s}{d+r} } \lambda_d(A_{k,l}^m) \\
& \geq & \theta^{s+d}J_{s,d}  \left( \frac{m+1}{m} \int f^{\frac{d}{d+r}} \lambda_d(dx) \right)^{s/d} \int_{\{f>0 \}} \tilde{f}_m^{\theta,\mu}(\tilde{f}_m + 2^{-m})^{-\frac{s}{d+r}} d\lambda_d .
\end{eqnarray*}
Finally, applying Fatou's Lemma  yields
\begin{equation*}
 \liminf_{n \rightarrow + \infty}  n^{s/d} \ \Vert X - \widehat{X}^{\alpha_n^{\theta,\mu}} \Vert_s^s \geq \theta^{s+d}  J_{s,d} \left(  \int_{\mathbb{R}^d} f^{\frac{d}{d+r}} d\lambda_d   \right)^{s/d} \int_{\{f>0\}} f_{\theta,\mu} f^{- \frac{s}{d+r}} d\lambda_d .
\end{equation*}
\end{proof}
\section{Upper estimate}
Let $ r,s > 0$. Let $(\alpha_n)_{n \geq 1}$ be an (asymptotically) $L^r(P)$ - optimal sequence of quantizers. In this section we  will provide some  sufficient conditions of $L^s(P)$-rate-optimality for the sequence  $(\alpha_n^{\theta,\mu})_{n \geq 1}$.
\vskip 0.4cm
\begin{defi}
Let $\theta>0,\ \mu \in \mathbb{R}^d$ and let  $P$ be a probability distribution such that  $P=f\cdot \lambda_d$. The couple $(\theta,\mu)$ is said $P$-\textbf{admissible} if 

\begin{equation} \label{admissible}
\{ f>0 \} \subset \mu(1-\theta) + \theta \{ f > 0 \} \qquad \lambda_d \textrm{-p.p.}
\end{equation}
\end{defi}

\begin{thm} \label{thm2} Let $ r,s \in (0,+\infty), s < r $ and let $X$ be a random variable taking  values in $\mathbb{R}^d$ with distribution $P$ such that $P = f \cdot \lambda_d$. Suppose that $(\theta,\mu)$ is $P$-admissible, for $\theta>0; \mu \in \mathbb{R^d}$,  and  $\mathbb{E}\vert X\vert ^{r+\eta} < \infty , \textrm{ for some } \eta >0 $. Let $(\alpha_n)_{n \geq 1} $ be an asymptotically $ L^r$-optimal sequence. If 
\begin{equation}
 \int_{\{ f > 0 \}} f_{\theta,\mu}^{\frac{r}{r-s}} f^{-\frac{s}{r-s}} d\lambda_d < +\infty 
\end{equation}
 then, $(\alpha_n^{\theta,\mu})_{n \geq 1}$ is $L^s(P)$-rate-optimal and  
\begin{equation}\label{limsup1}
\limsup_{n \rightarrow +\infty} n^{s/d} \ \Vert X - \widehat{X}^{\alpha_n^{\theta,\mu}}  \Vert_s^s \leq \theta^{s +d}\left(Q_r(P)\right)^{s/r} \left( \int_{\{ f > 0 \}} f_{\theta,\mu}^{\frac{r}{r-s}} f^{-\frac{s}{r-s}} d\lambda_d \right)^{1-\frac{s}{r}}.
\end{equation}
\end{thm}
\vskip 0.4cm
\begin{rem}
 Note that if $\theta=1$ and $\mu=0$ then  $$\int_{\{ f > 0 \}} f_{\theta,\mu}^{\frac{r}{r-s}} f^{-\frac{s}{r-s}} d\lambda_d = \int_{\{ f > 0 \}} f^{\frac{r}{r-s}} f^{-\frac{s}{r-s}} d\lambda_d = \int_{\{ f > 0 \}} f d\lambda_d  = 1.$$
In this case the theorem is trivial since $\ \Vert X - \widehat{X}^{\alpha_n}  \Vert_s \leq \Vert X - \widehat{X}^{\alpha_n}  \Vert_r.$
\end{rem}

\begin{proof}[\textbf{Proof}] 
Let  $P^{\theta}$ denotes the distribution of the random variable   $\theta X $. $P^{\theta}$ is absolutely continuous with respect to $\lambda_d$, with $p.d.f \ g_{\theta}(x) = \theta^{-d} f(\frac{x}{\theta}) $. 

For every $n \geq 1$,
\begin{eqnarray*}
n^{s/d} \ \Vert X- \widehat{X}^{\alpha_n^{\theta,\mu}}  \Vert_s^{s}
& = & n^{s/d} \int_{\mathbb{R}^d} d(x, \alpha_n^{\theta,\mu})^s dP(x) \\ & = &
n^{s/d} \int_{\{f>0 \}} \min_{a \in \alpha_n} \vert x- \mu(1-\theta) - \theta a \vert^s f(x) d\lambda_d(x).
\end{eqnarray*}
Making  the change of variable $z := x - \mu(1-\theta)$  yields
\setlength\arraycolsep{1pt}
\begin{eqnarray}
n^{s/d} \ \Vert X- \widehat{X}^{\alpha_n^{\theta,\mu}}  \Vert_s^{s}
& = & n^{s/d} \int_{\{ f>0 \} - \mu(1-\theta) } d(z,\theta \alpha_n)^s f(z + \mu(1-\theta)) d\lambda_d(z) \nonumber  \\ & \leq &
n^{s/d} \int_{\theta \{ f>0\} } d(z,\theta \alpha_n)^s f(z + \mu(1-\theta)) g_{\theta}^{-1}(z)dP^{\theta}(z) \\ & \leq &
n^{s/d} \left(\int_{\mathbb{R}^d} d(z,\theta \alpha_n)^r  dP^{\theta}(z)\right)^{s/r} \left( \int_{\theta \{ f>0\}} \big(f(z + \mu(1-\theta) )g_{\theta}^{-1}(z) \big)^{\frac{r}{r-s}} dP^{\theta}(z) \right)^{\frac{r-s}{r}}  \nonumber \\ & \leq &
\left(n^{r/d} \Vert \theta X - \widehat{\theta X}^{\theta \alpha_n} \Vert_r^r \right)^{s/r} \left( \int_{\theta \{ f>0\} }  f(z + \mu(1-\theta))^{\frac{r}{r-s}} g_{\theta}^{-\frac{s}{r-s}}(z) d\lambda_d(z) \right)^{\frac{r-s}{r}} \nonumber
\end{eqnarray}
where we used the $P$-admissibility of $(\theta,\mu)$  in the first inequality. 
The second  inequality  derives from \textit{H\"{o}lder} inequality applied with  $ p=r/s >1$ and $q=1-s/r$.

Moreover
\begin{equation}  \label{eq1pthm2}
\Vert \theta X - \widehat{\theta X}^{\theta \alpha_n} \Vert_r^r = \mathbb{E} \big( \min_{a \in \alpha_n} \vert \theta X - \theta a  \vert^r \big) = \theta^r \Vert X- \widehat{X}^{\alpha_n}  \Vert_r^r. 
\end{equation}
Then
 $$n^{s/d} \ \Vert X- \widehat{X}^{\alpha_n^{\theta,\mu}}  \Vert_s^{s} \leq \theta^s \left(n^{r/d} \Vert X - \widehat{X}^{\alpha_n} \Vert_r^r \right)^{s/r} \left( \int_{\theta \{ f>0\}}  f(z + \mu(1-\theta))^{\frac{r}{r-s}} g_{\theta}^{-\frac{s}{r-s}}(z) d\lambda_d(z) \right)^{\frac{r-s}{r}} .$$
 Owing to  the asymptotically $L^r(P)$-optimality of   $(\alpha_n)$  and   making again the change of variable   $  x := z / \theta $ yields
\begin{eqnarray*}
\limsup_{n \rightarrow +\infty} n^{s/d} \ \Vert X- \widehat{X}^{\alpha_n^{\theta,\mu}}  \Vert_s^{s}
& \leq & \theta^s \left(Q_r(P)\right)^{s/r} \left(\theta^{\frac{ds}{r-s}} \int_{\theta \{ f>0\}} f(z + \mu(1-\theta)))^{\frac{r}{r-s}} f(z/\theta)^{-\frac{s}{r-s}} d\lambda_d(z) \right)^{\frac{r-s}{r}}  \\ & = &
\theta^s \left(Q_r(P)\right)^{s/r} \left(\theta^{\frac{rd}{r-s}} \int_{\{ f>0\} } f_{\theta,\mu}(x)^{\frac{r}{r-s}} f(x)^{-\frac{s}{r-s}} d\lambda_d(x) \right)^{\frac{r-s}{r}} \\ & = &
\theta^{s+d} \left(Q_r(P)\right)^{s/r} \left(\int_{\{ f>0\} } f_{\theta,\mu}(x)^{\frac{r}{r-s}} f(x)^{-\frac{s}{r-s}} d\lambda_d(x) \right)^{\frac{r-s}{r}}.
\end{eqnarray*}
\setlength\arraycolsep{2pt}
\end{proof}

When $s>r$, the next theorem provides a less accurate  asymptotic upper bound than the previous one since, beyond the restriction on the distribution of $X$, we  need now  the sequence $(\alpha_n)$ to be   (exactly) $L^r(P)$-optimal. 
\vskip 0.4cm
\begin{thm} \label{thm3}   Let $ r,s \in (0,+\infty),\ s > r,\ \theta >0 $ and  let  $X$ be a random variable taking values in $\mathbb{R}^d$ with distribution  $P$  such that $P=f\cdot \lambda_d$. Suppose that $\mathbb{E}\vert X\vert ^{r+\eta} < \infty  \textrm{ for some } \eta >0 $ and $\ P_{\theta,\mu} \ll P$ \big(i.e $P_{\theta,\mu}$ is absolutely continuous with respect to $P$ \big) for some $\mu \in \mathbb{R}^d$. Let $(\alpha_n)_{n \geq 1} $ be an $ L^r(P)$- optimal sequence of quantizers  and suppose that the maximal function (see $(\ref{defoncmax})$) satisfies
\begin{equation} \label{foncmax} 
 \psi_b^{s/(d+r)} \in \ L^1(P_{\theta,\mu})  \ \textrm{ for some } \ b \in (0,1/2).
\end{equation}
Then, 
\begin{equation} \label{limsup2} 
\limsup_{n} n^{s/d} \ \Vert X- \widehat{X}^{\alpha_n^{\theta,\mu}} \Vert_s^s  \leq \theta^{s+d} C(b) \int f_{\theta,\mu}f^{-\frac{s}{d+r}} d\lambda_d < +\infty
\end{equation}
where $C(b)$ is a positive  real constant not depending on $\theta$ and $n$.
\end{thm}
\vskip 0.4cm
Notice that  this theorem does not require that  $(\theta,\mu)$ is  $P$-admissible.

\begin{proof}[\textbf{Proof}] \ One deduces from differentiation of measures that 
$$ f^{-\frac{s}{d+r}} \ \leq \ \psi_b^{\frac{s}{d+r}} \qquad P_{\theta,\mu}\textrm{-a.s}. $$ 
Then, under Assumption $(\ref{foncmax})$,
$$ \int f^{-\frac{s}{d+r}} dP_{\theta,\mu} = \int f_{\theta,\mu}f^{-\frac{s}{d+r}} d\lambda_d < +\infty.$$
  For all $ n \geq 1$,
\begin{eqnarray*}
n^{s/d} \ \Vert X- \widehat{X}^{\alpha_n^{\theta,\mu}} \Vert_s^s 
& = & n^{s/d}  \int_{\mathbb{R}^d} d(z,\alpha_n^{\theta,\mu})^s f(z) d\lambda_d(z) \\ & = &
n^{s/d} \theta^s \int_{\mathbb{R}^d} \min_{a \in \alpha_n} \vert (z-\mu)/\theta+\mu -  a \vert^s f(z) d\lambda_d(z) 
\end{eqnarray*}
We make the change of variable $x := (z-\mu)/ \theta + \mu$. Then 
\begin{eqnarray*} 
n^{s/d} \ \Vert X- \widehat{X}^{\alpha_n^{\theta,\mu}} \Vert_s^s 
& = & n^{s/d} \theta^{s+d} \int_{\mathbb{R}^d} d(x,\alpha_n)^s f(\mu +\theta(x-\mu)) d\lambda_d(x) \\ & = &
n^{s/d}\theta^{s} \int_{\mathbb{R}^d} d(x,\alpha_n)^s dP_{\theta,\mu}(x).
\end{eqnarray*}
Besides, the following inequalities  are established in $\cite{GLP}$ :
\begin{eqnarray*} 
 \limsup_{n} n^{s/d} d(\cdot,\alpha_n)^s & \leq & C(b) f^{-\frac{s}{d+r}} \\
 \textrm{and } \hspace{1.0cm} n^{s/d}d(\cdot,\alpha_n)^s  & \leq & C(b) \psi_b^{\frac{s}{d+r}} \qquad P\textrm{-a.s} \ ( \textrm{hence } P_{\theta,\mu} \textrm{-a.s., since } P_{\theta,\mu} \ll P) .
\end{eqnarray*}
Under Assumption  $(\ref{foncmax})$  we can apply  the Lebesgues dominated convergence theorem  to the above inequalities, which yields 
\begin{eqnarray*}
 \limsup_{n} n^{s/d} \int d(x,\alpha_n)^s dP_{\theta,\mu}(x)
& \leq &  \int \limsup_{n} n^{s/d} d(x,\alpha_n)^s dP_{\theta,\mu}(x)  \\
& \leq & C(b) \int f^{-\frac{s}{d+r}} dP_{\theta,\mu}(x). \\
& = & \theta^{d} C(b) \int f_{\theta,\mu}(x)f^{-\frac{s}{d+r}}(x) d\lambda_d(x).
 \end{eqnarray*}
 \end{proof}

For a given distribution, Assumption $(\ref{foncmax})$ is not easy to verify. But when $s \not= r+d$, the lemma and corollaries below provide a sufficient condition so that Assumption  $(\ref{foncmax})$ is  satisfied. The next subject extends the results obtained in  $(\cite{GLP})$. For  further  details  we then  refer to  $(\cite{GLP})$.

Let $P=f \cdot \lambda_d$  be an absolutely continous distribution. Let $r,s \in (0, +\infty)$ and $(\theta,\mu)$ be a $P$-admissible couple of parameters. We will need the following hypotheses:  
\begin{itemize}
\item[$\textbf{(H1)}$]   for all $M>0$, 
\begin{equation}
 \sup_{z \in B(0,M)} \frac{f(\mu + \theta(z-\mu))}{f(z)} \mbox{\bf{1}}_{\{f(z)>0 \}} < +\infty .
 \end{equation}
\item[$\textbf{(H2)}$]   There exists  $b \in (0,1/2),\ M \in (0,+\infty)$ such that
\begin{equation}
 \int_{B(0,M)^{c}} \left( \sup_{t \leq 2b\vert x \vert} \frac{\lambda_d(B(x,t))}{P(B(x,t))} \right)^{s/(d+r)} dP_{\theta,\mu} < + \infty .
 \end{equation}
\item[$\textbf{(H3)}$]   $ \lambda_d(\cdot \cap \textrm{supp}(P)) \ll P $ and  $\textrm{supp}(P)$  is a finite union of closed convex sets.
\end{itemize}

\begin{lem} \label{lem3.2}
Let  $ P = f \cdot \lambda_d$ and $r>0$ such that $\int \vert x \vert^r P(dx) < +\infty$. Assume $(\alpha_n)_{n \geq 1}$ is a sequence of quantizers such that $\int d(x,\alpha_n)^r dP \rightarrow 0$.
 Let $(\theta,\mu)$ be a  $P$-admissible  couple of parameters  for which  $(\textbf{H1})$ holds.
\begin{itemize}
\item[$(a)$]  If $p \in (0,1)$  then  for every $ b>0, \ \psi_b^p \in L_{loc}^1(P_{\theta,\mu})$. 

\item[$(b)$] If $p \in (1,+\infty]$ and if furthermore $\textbf{(H3)}$ holds then for every $b >0$,   $$ f^{-p} \in L_{\textrm{loc}}^1(P) \Longrightarrow  \psi_b^p \in L_{\textrm{loc}}^1(P_{\theta,\mu}) .$$
\end{itemize}
\end{lem}

\begin{proof}[$\textbf{Proof}$]   It follows from the $P$-admissibility of $(\theta,\mu)$ that $$ P_{\theta,\mu}(dz) = \theta^d f(\mu + \theta(z-\mu)) \lambda_d(dz) = g_{\theta}(z) P(dz), $$ where $g_{\theta}(z)= \theta^d \ \frac{f(\mu + \theta(z-\mu))}{f(z)} \mbox{\bf{1}}_{\{f(z)>0 \}}$. Then $g_{\theta}$ is locally  bounded   by $ \ (\textbf{H1})$.

$(a)$ If $p \in (0,1)$, it  follows from Lemma 1 in  $\cite{GLP}$ that  $\psi_b^p \in L_{\textrm{loc}}^1(P)$. Hence $ \psi_b^p \in L_{\textrm{loc}}^1(P_{\theta,\mu})$  since $g_{\theta}$ is locally bounded. 

$(b)$ If $p \in (1,+\infty)$ it follows from  Lemma 2  in $\cite{GLP}$ that  if  $ f^{-p} \in L_{\textrm{loc}}^1 (P)$  then  $\psi_b^p \in L_{\textrm{loc}}^1(P_{\theta,\mu})$  since $g_{\theta}$ is locally  bounded. 

\end{proof}

\begin{cor}  \label{cor3.1} 
(Distributions with unbounded supports ) \ 
Let  $r>0, \ s \in (0, +\infty), \ s \not= r+d$  and let  $X$ be a random variable with probability measure  $P=f \cdot \lambda_d$  such that   $E \vert X \vert^{r+ \eta} < +\infty$ for some  $\eta >0$. Let  $(\theta,\mu)$  be  $P$-admissible and suppose that $\textbf{(H1)}, \ \textbf{(H2)}$ hold.  
\begin{itemize}
\item[$(a)$] If $s \in (0,r+d)$ then Assumption  $(\ref{foncmax})$ of Theorem $\ref{thm3}$  holds true.
\item[$(b)$] If  $s \in (r+d,+\infty)$,  and if furthermore,  $ \textbf{(H3)} $ holds and $f^{-\frac{s}{r+d}} \in L_{loc}^1(P)$ then Assumption  $(\ref{foncmax})$ of Theorem $\ref{thm3}$  holds true.
\end{itemize}
\end{cor}
\vskip 0.4cm
\begin{proof}[\textbf{Proof}] \ Let $x_0 \in \textrm{supp}(P)$. We know from $\cite{DGLP}$ that $d(x_0,\alpha_n) \rightarrow 0$. Then following the lines of  the proof of Corollary 2 in $\cite{GLP}$ one has  for $\vert x \vert> N= \vert x_0 \vert + \sup_{n \geq 1} d(x_0,\alpha_n) $,  $\ d(x,\alpha_n) \leq 2 \vert x \vert $  for every $n \geq 1$.  Thus for every $b>0, \ x \in B(0,N)^c$,
\begin{equation*}  \label{ineq_lie_H2}
 \psi_b(x) \leq \sup_{t \leq 2b \vert x \vert} \frac{\lambda_d(B(x,t))}{P(B(x,t))}. 
\end{equation*} 
Now, coming back to the core of our  proof, it  follows from $\textbf{(H2)}$ that  (for $b$ coming from $\textbf{(H2)}$),  $$ \int_{B(0,M \vee N )^{c}} \psi_b^{s/(d+r)} dP_{\theta,\mu} < +\infty .$$
Since   $$ \int  \psi_b^{s/(d+r)} dP_{\theta,\mu} = \int_{B(0,M \vee N)} \psi_b^{s/(d+r)} dP_{\theta,\mu} + \int_{B(0,M \vee N)^{c}} \psi_b^{s/(d+r)} dP_{\theta,\mu}, $$ 
it remains to show that  the first term in the right hand side of this last  equality  is finite.

$(a)$ If $s \in (0,r+d)$ it follows from Lemma $\ref{lem3.2},\ (a)$ that the first term in the right hand side of the above equality is finite. As a consequence, $\psi^{\frac{s}{r+d}} \in L^1(P_{\theta,\mu})$. 

$(b)$ If $s > r+d$,  the first term in the right hand side of the  above equality  still  finite owing to Lemma $\ref{lem3.2}, \ (b)$. Consequently, Assumption  $(\ref{foncmax})$ of Theorem $\ref{thm3}$  holds true provided $\textbf{(H3)}$ holds and $f^{-\frac{s}{r+d}} \in L_{loc}^1(P)$. 
 
\end{proof}
We next give two  useful criterions  ensuring that  Hypothesis $\textbf{(H2)}$ holds.  The first one is useful for distributions with radial tails and the second one for distributions which does not  satisfy this last  assumption.

\begin{crit}  \ Let  $X$ be a random variable with probability measure   such that $ P = f \cdot \lambda_d$ and  $E | X |^{r+ \eta} < +\infty$ for some $\eta >0$.   \\
$(a)$ \  Let  $r,s>0$  and  $f = h(\vert  \cdot  \vert)$  on $B_{\vert \cdot \vert}(0,N)^{c}$ with $h : (R,+\infty) \rightarrow  \mathbb{R}_{+},\ R \in \mathbb{R}_{+}$, a decreasing function  and $\vert \cdot \vert$ any norm on $\mathbb{R}^{d}$. 
  Suppose  that $ (\theta,\mu) $ is  a couple of $P$-admissible  parameters such that 

\begin{equation}  \label{asscor1}
 \int f(cx)^{-\frac{s}{d+r}} dP_{\theta,\mu}(x)  < +\infty  
 \end{equation}
for some  $c>1$.  Then  $\textbf{(H2)}$ holds. \\
$(b)$ \ Let  $r,s>0 $.  If $ \textrm{supp}(P) \subset [R_0, +\infty)$  for some $ \ R_0 \in \mathbb{R}$ and $f_{|(R_0^{'}, +\infty)}$  decreasing  for $R_0^{'} \geq R_0$.  Assume furthermore that  $ (\theta,\mu) $ is  a couple of $P$-admissible  parameters such that  $(\ref{asscor1})$ is satisfied for some  $c>1$.  Then Hypothesis  $\textbf{(H2)}$ holds.
\end{crit}
Note that $(b)$ follows from $(a)$, for $d=1$,  and that  $(a)$ is simply deduced from the proof of  Corollary 3 in $\cite{GLP}$ since it has been  showed  that for $b \in (0,1/2), \ M:=N/(1-2b) $ one has for every $x \in B(0,M)^c$,
$$ \sup_{t \leq 2b \vert x \vert}  \frac{\lambda_d(B(x,t))}{P(B(x,t))} \leq  \frac{1}{f(x(1+2b))}.$$

\vskip 0.4cm
\begin{crit}  \ Let $r,s>0,\ P=f\cdot \lambda_d$ and  $\int \vert x \vert^{r+ \eta} P(dx) < +\infty$ for some $\eta >0$. Let   $(\theta,\mu)$ be a $P$-admissible  couple such that 
\begin{equation} \label{assert1}
 \sup_{z \not= 0 } \frac{f(\mu + \theta(z-\mu))}{f(z)} \mbox{\bf{1}}_{\{f(z)>0 \}} < +\infty . 
\end{equation} 
Assume furthermore that $$ \inf_{x \in \textrm{supp}(P), \rho>0} \frac{\lambda_d(\textrm{supp}(P) \cap B(x,\rho))}{\lambda_d(B(x,\rho))} >0$$
and that $f$ satisfies   the local growth control assumption : there exists real numbers $\varepsilon \geq 0, \ \eta \in (0,1/2), \ M, C >0 $ such that
$$ \forall x,y \in \textrm{supp}(P), \ \vert x \vert \geq M, \ \vert y-x \vert \leq 2 \eta \vert x \vert \  \Longrightarrow \ f(y) \geq C f(x)^{1+\varepsilon}. $$
If   
\begin{equation} \label{Hypoth2}
\int f(x)^{-\frac{s(1+\varepsilon)}{d+r}} dP(x) < +\infty, 
\end{equation}
  then $\textbf{(H2)}$ holds. If in particular $f$ satisfies the local growth control assumption for  $\varepsilon=0$ or for every $\varepsilon \in (0,\underline{\varepsilon}]$, with $\underline{\varepsilon}>0$, and if  
 $$\int f(x)^{-\frac{s}{d+r}} dP(x) = \int_{\{ f>0 \}} f(x)^{1-\frac{s}{d+r}} d\lambda_d(x) < +\infty$$
  then Hypothesis $\textbf{(H2)}$ holds.
\end{crit}
Notice that Hypothesis $(\ref{assert1})$ can be relaxed if we suppose that $f(x)^{-\frac{s(1+\varepsilon)}{d+r}}  \in L^1(P_{\theta,\mu})$ instead of $(\ref{Hypoth2})$. The criterion follows from Corollary 4 in $\cite{GLP}$.

\section{Toward a necessary and sufficient condition for $L^s(P)$-rate optimality when $s>r$}
Before dealing with examples, let us make some comments about  inequalities  $(\ref{liminf})$ and $(\ref{limsup2})$. Note first that the moment assumption  $ \ \mathbb{E} \vert X \vert^{r+\eta} < +\infty \textrm{ for some  } \eta >0$, ensure that  $\int_{\mathbb{R}^d} f^{\frac{d}{d+r}} d\lambda_d < + \infty$\ (cf \cite{GL}). Consequently, if  $\int f_{\theta,\mu} f^{-\frac{s}{d+r}} d\lambda_d = +\infty $ one derives from inequality $(\ref{liminf})$ that  $$\lim_{n} \Vert X - \widehat{X}^{\alpha_n^{\theta,\mu}} \Vert_s ^s = + \infty. $$ 
Then the sequence  $(\alpha_n^{\alpha,\mu})_{n \geq 1}$ is not  $L^s$-rate-optimal.

 On the other hand if $\int f_{\theta,\mu} f^{-\frac{s}{d+r}} d\lambda_d  < +\infty $ one derives from Inequality $(\ref{limsup2})$  that  $(\alpha_n^{\theta,\mu})_{n \geq 1}$ is  $L^s$-rate-optimal. For $s>r$, this leads to a necessary and sufficient condition  so that the sequence $ (\alpha_n^{\theta,\mu})_{n \geq 1}$  (in particular the sequence $(\alpha_n)_{n \geq 1}$ by taking $\theta=1$ and $\mu=0$) is  $L^s$-rate-optimal.
\vskip 0.4cm
\begin{rem}  \label{cor1}  Let $\mu \in \mathbb{R}^d $, $\theta,r >0,\ s>r$ and let $P$ be a probability distribution such that $P=f \cdot \lambda_d$. Assume $(\theta,\mu)$ is $P$-admissible.  Let  $(\alpha_n)_{n \geq 1}$ be an  $L^r(P)$-optimal sequence of $n$-quantizers  and suppose that Assumption  $(\ref{foncmax})$ of  Theorem $\ref{thm3}$ holds true. Then 
\begin{equation} \label{eqcor1}
(\alpha_n^{\theta,\mu})_{n \geq 1}  \textrm{ is  $L^s$-rate-optimal } \Longleftrightarrow  \int f_{\theta,\mu} f^{-\frac{s}{d+r}} d\lambda_d < +\infty .
\end{equation}
\end{rem}

\vskip 0.4cm
\begin{rem} \label{rem3.1} 
If $s<r$, the inequality $(\ref{limsup1})$ provides a sufficient condition so that  the sequence  $(\alpha_n^{\theta,\mu})_{n \geq 1}$ is   $L^s$-rate-optimal, which  is :  $ \int f_{\theta,\mu}^{\frac{r}{r-s}} f^{-\frac{s}{r-s}} d\lambda_d  < +\infty$ (always satisfied by $(\alpha_n)_{n \geq 1}$ itself).
\end{rem}

Now, for  $s \not= r$, is it possible to find a $ \ \theta = \theta^{\star}$ for  which the sequence $(\alpha_n^{\theta,\mu})_{n \geq 1}$ is  asymptotically $L^s(P)$-optimal? (when $s<r$ this is the only question of interest since we know that $(\alpha_n)_{n \geq 1}$ is $L^s(P)$-rate-optimal for every $s<r$).

For a fixed  $r,b$ and $\mu$, we can write from inequalities  $(\ref{limsup1})$ and  $(\ref{limsup2})$ :
\begin{equation}  \label{ineq_princ}
\limsup_n n^{s/d} \ \Vert X-\widehat{X}^{\alpha_n^{\theta,\mu}} \Vert_s^s \leq Q_{r,s}^{\textrm{Sup}}(P,\theta) 
\end{equation} 
with 
$$
Q_{r,s}^{\textrm{Sup}}(P,\theta)=\left \{\begin{array}{ll}
\theta^{s +d}\left(Q_r(P)\right)^{s/r} \left( \int_{f>0} f_{\theta,\mu}^{\frac{r}{r-s}} f^{-\frac{s}{r-s}} d\lambda_d \right)^{1-\frac{s}{r}} & \textrm{ if } s<r\\

\theta^{s+d} C(b) \int f_{\theta,\mu}f^{-\frac{s}{d+r}} d\lambda_d  & \textrm{ if } s>r.
  \end{array} \right.
$$
\vskip 0.4cm
 One knows that for a  given $s>0$, we have  for all  $  \  n \geq 1$, 
$$ e_{n,s}^s(X)  \leq   \Vert X - \widehat{X}^{\alpha_n^{\theta,\mu}} \Vert_s ^s. $$
Then for every $\theta>0$,    $$ Q_s(P)  \leq  Q_{r,s}^{\textrm{Sup}}(P,\theta). $$
 Consequently for a fixed  $s>0$,   in order to have the best  estimation of  Zador's constant in  $L^s$, we must minimize over $\theta $,   the quantity  $Q_{r,s}^{\textrm{Sup}}(P,\theta)$.  In that way, we may hope to reach the sharp rate of convergence in Zador Theorem and so construct an  asymptotically $L^s$-optimal sequence.

For $\mu$ well chosen, the examples below show that, for the  Gaussian and the  exponential distribution, the minimum $\theta^{\star}$ exists  and the sequence  $(\alpha_n^{\theta^{\star}, \mu})_{n \geq 1}$ satisfies the empirical measure theorem and is suspected to be  asymptotically $L^s$-optimal.

\section{Examples}
Let  $(\alpha_n)_{n\geq 1}$ be  an $L^r(P)$-optimal sequence of quantizers  for a given probability distribution $P$, and consider the sequence $(\alpha_n^{\theta,\mu})_{n \geq 1}$. For a fixed $\mu$ and $s$, we try  to solve the following minimization problem 

\begin{equation} \label{eqmin}
 \theta^{\star} = \arg\min_{\theta >0} \big \{ Q_{r,s}^{\textrm{Sup}}(P,\theta), \ (\alpha_n^{\theta,\mu})_{n \geq 1} \ L^s(P)\textrm{-rate-optimal} \big \}.
\end{equation}

In all  examples,   $C$  will  denote a generic real constant (not depending on $\theta$)  which may change from line to line.  The choice of  $\mu$  depends on the probability measure and it is not clear how to choose it.  In practice,  we shall set   $\mu = \mathbb{E}(X)$  when $X$  is  a symmetric random variable otherwise   we will usually set   $\mu = 0$. 

 \subsection{The multivariate Gaussian distribution}
 \subsubsection{Optimal dilatation and contraction}  
\begin{prop}
Let $r,s >0$ and let $P=\mathcal{N}(m;\Sigma), \ m \in \mathbb{R}^d, \Sigma \in \mathcal{S}^{+}(d,\mathbb{R})$.
\begin{enumerate}
\item[\bf{(a)}] If  $s \in (r,r+d) \cup (r+d,+\infty)$, the sequence  $(\alpha_n^{\theta,m})_{n \geq 1}$ is $L^s(P)$-rate-optimal iff $ \ \theta \in \big(\sqrt{s/(d+r)},+\infty \big)$  and  $$ \theta^{\star} = \sqrt{(s+d)/(r+d)} \ \in (1,+\infty) $$ is the unique solution of $(\ref{eqmin})$ on the set  $\big(\sqrt{s/(d+r)},+\infty \big).$
\item[\bf{(b)}] If $s \in (0,r)$, the sequence   $(\alpha_n^{\theta,m})_{n \geq 1}$ is $L^s(P)$-rate-optimal if  $ \ \theta \in \big(\sqrt{s/r},+\infty \big)$ \\  and $$ \theta^{\star} = \sqrt{(s+d)/(r+d)} \ \in (0,1) $$ is the unique solution of  $(\ref{eqmin})$ on the set  $\big(\sqrt{s/r},+\infty \big).$
\end{enumerate}
\end{prop}

\begin{proof}[\textbf{Proof}]
Since the multivariate Gaussian distribution is symmetric, one sets $\mu = m$. Keep in mind that the probability  density function $f$ of  $P$  is given for every $x \in \mathbb{R}^d$ by,
$$ f(x)= \big( (2\pi)^d \textrm{det}\ \Sigma \big)^{-\frac{1}{2}} e^{-\frac{1}{2} (x-m)'\Sigma^{-1}(x-m)}. $$
Note first that Hypothesis $\textbf{(H1)}$ is obviously satisfied from the continuity of $\frac{f(m+\theta(z-m))}{f(z)} \mbox{\bf{1}}_{\{f(z)>0 \}} $ on every $\bar{B}(0,M),\ M>0$.

 $\mbox{\bf{(a)}} \ $ Let $s \in (r,d+r) $. For every $\theta>0, \mu \in \mathbb{R}^d$, the couple $(\theta,\mu)$ is $P$-admissible ($f>0$) and  $f$ is radial since $f(x) = \varphi(\vert x-m \vert_{\Sigma}) \textrm{ with } \varphi : (0,+\infty) \longmapsto \mathbb{R}_{+} $ defined by $$\varphi( \xi ) = \big((2\pi)^{d} \textrm{det} \Sigma \big)^{-1/2} \exp (-\frac{1}{2} \vert \xi \vert^2 ), \  \textrm{ with }\vert x \vert_{\Sigma} = \vert \Sigma^{-\frac{1}{2}} x \vert. $$ 
 Let $\theta> \sqrt{s/(r+d)}$. Then  Assumption $(\ref{asscor1})$ holds for every $c \in (1, \theta  \sqrt{\frac{r+d}{s}})$. Consequently,  it follows from   Corollary $\ref{cor3.1}, (a)$  that Assumption  $(\ref{foncmax})$ of Theorem $\ref{thm3}$ holds. 

If $s>d+r$,  the required additional hypotheses $\textbf{(H3)}$ and $f^{-\frac{s}{r+d}} \in L_{loc}^{1}(P)$ are clearly satisfied since $P=f \cdot \lambda_d$ (and $f^{-1}$ is continuous ensuring that $ \lambda_d(\cdot \cap \textrm{supp}(P)) \ll P$) and $f^{-\frac{s}{r+d}} $ is continuous on every $\bar{B}(0,M),\ M>0$. Then it follows from  Corollary $\ref{cor3.1},  (b)$ that Assumption  $(\ref{foncmax})$ of Theorem $\ref{thm3}$ holds.


In the other hand
\begin{eqnarray*}
 \int_{\mathbb{R}^d} f_{\theta,m}(x)f(x)^{-\frac{s}{d+r}} dx
 & = & \int_{\mathbb{R}^d} f(m + \theta(x-m))f(x)^{-\frac{s}{d+r}} dx \\
 & = & C \ \int_{\mathbb{R}^d}  e^{-\frac{1}{2} ( \theta^2 - \frac{s}{d+r} )(x-m)'\Sigma^{-1}(x-m)} dx  
\end{eqnarray*}
 so that  $$ \int_{\mathbb{R}^d} f_{\theta,m}(x)f(x)^{-\frac{s}{d+r}} < +\infty \quad \textrm{ iff } \quad  \theta> \sqrt{\frac{s}{d+r}}.$$

Now we are in position to solve the problem $(\ref{eqmin})$. Let  $\theta \in \big(\sqrt{s/(d+r)},+\infty \big)$,

\begin{eqnarray*}
\theta^{s+d} \int_{\mathbb{R}^d} f_{\theta,m}(x)f(x)^{-\frac{s}{d+r}} dx
& =  & \big((2 \pi)^d det\ \Sigma \big)^{-\frac{1}{2}(1-\frac{s}{d+r})} \ \theta^{s+d} \int_{\mathbb{R}^d}  e^{-\frac{1}{2} ( \theta^2 - \frac{s}{d+r} )(x-m)'\Sigma^{-1}(x-m)} dx \\
& =  &  \big((2 \pi)^d det\ \Sigma \big)^{-\frac{s}{d+r}}\ \theta^{s+d} \left( \theta^2 - \frac{s}{d+r} \right)^{-\frac{d}{2}}. 
\end{eqnarray*}
 For  $ \theta \in  \big(\sqrt{s/(d+r)},+\infty \big)$, we want to  minimize the function  $h$ defined by 
 $$ h(\theta)= \theta^{s+d} \left( \theta^2 - \frac{s}{d+r} \right)^{-\frac{d}{2}}. $$
 The function  $h$ is differentiable  on $\big(\sqrt{s/(d+r)}, + \infty \big)$ with derivative    
$$ h'(\theta) = s \theta^{d+s-1}\left( \theta^2 - \frac{s}{d+r} \right)^{-1-d/2} \left( \theta^2 - \frac{s+d}{r+d} \right).$$
One easily checks that $h$ reaches its unique minimum on  $\big(\sqrt{s/(d+r)},+\infty \big) $ at $ \theta^{\star} = \sqrt{(s+d)/(r+d)}.$

$\mbox{\bf{(b)}}$ Let $s<r$ and consider the  inequality  $(\ref{limsup1})$. We get
\begin{equation*}
 \int f_{\theta,m}^{\frac{r}{r-s}}(x) f^{-\frac{s}{r-s}}(x) dx =  C \ \int_{\mathbb{R}^d}  e^{-\frac{1}{2} \frac{r}{r-s} ( \theta^2 - \frac{s}{r} ) (x-m)'\Sigma^{-1}(x-m)} dx. 
 \end{equation*}
So if  $ \ \theta \in \big(\sqrt{s/r},+\infty \big) \ \textrm{ then } \int f_{\theta,m}^{\frac{r}{r-s}}(x) f^{-\frac{s}{r-s}}(x) dx < +\infty. $ This proves the first assertion.

To prove the second assertion, let  $ \ \theta \in \big(\sqrt{s/r},+\infty \big) $. Then
\begin{eqnarray*}
\theta ^{d+s} \left( \int f_{\theta,m}^{\frac{r}{r-s}}(x) f^{-\frac{s}{r-s}}(x) dx \right)^{1-\frac{s}{r}}  
& = & C \ \theta^{s+d} \left(\int_{\mathbb{R}^d}  e^{-\frac{1}{2} \frac{r}{r-s} ( \theta^2 - \frac{s}{r} ) (x-m)'\Sigma^{-1}(x-m)} dx \right)^{1-\frac{s}{r}}\\
& =&  C \ \theta^{s+d} \left( \theta^2 - \frac{s}{r} \right)^{-\frac{d}{2r}(r-s)}.
\end{eqnarray*}
We  proceede as before  setting  $$ h(\theta) = \theta^{\alpha} \left( \theta^2 - \frac{s}{r} \right)^{\beta}   , \textrm{ with } \alpha = d+s \textrm{ and } \beta = -\frac{d}{2r}(r-s). $$
For all  $ \ \theta \in \big(\sqrt{s/r},+\infty \big) $, $$ h'(\theta) =  \theta^{\alpha -1 } \left( \theta^2 - \frac{s}{r} \right)^{\beta -1} \left( (\alpha +2\beta) \theta^2 - \frac{\alpha s}{r} \right). $$
The sign of  $h'$ depends on the  sign of  $\left( (\alpha +2\beta) \theta^2 - \frac{\alpha s}{r} \right)$.  Moreover  $\alpha +2 \beta = \frac{s}{r}(d+r) >0 $ then  $h'$ vanishes at $\theta^{\star} = \sqrt{(s+d)/(r+d)}$, is negative on the set $\big(\sqrt{s/r},\theta^{\star}\big)$ and positive on $\big(\theta^{\star},+\infty \big)$. Therefore $h$ reaches its  minimum on $\big(\sqrt{s/r},+\infty \big) $ at the  unique point   $\theta^{\star}.$

\end{proof}

\begin{rem} Let $X \sim \mathcal{N}(m;\Sigma).$ \\
 If $s<r,\textrm{ then } \ \theta^{\star} < 1$. Hence, $\ (\alpha_n^{\theta^{\star},m})_{n \geq 1}$ is a contraction of  $(\alpha_n)_{n \geq 1}$ with scaling number  $\theta^{\star}$ and translating number $m$. 
In the  other hand, if $s > r, \textrm{ then } \ \theta^{\star} > 1$. In this case the sequence  $\ (\alpha_n^{\theta^{\star},m})_{n \geq 1}$ is a dilatation of $(\alpha_n)_{n \geq 1}$ with scaling number  $\theta^{\star}$ and translating number $m$. Also note  that $\theta^{\star}$ does not  depend on the covariance matrix $ \Sigma $.
\end{rem}

What we do  expect from the  resulting sequence $(\alpha_n^{\theta^{\star},m})_{n \geq 1}$ ? The proposition below shows that it satisfies the empirical measure theorem (keep in mind that this theorem is satisfied by asymptotically optimal quantizers although the converse is not true in general).
\vskip 0.4cm
\begin{prop} \label{prop_empthm1}
 Let $r,s>0$ and let $ P = \mathcal{N}(m;\Sigma)$. Assume  $(\alpha_{n})_{n \geq 1} $ is asymptotically  $L^r(P)$-optimal. Then the sequence  $ (\alpha_{n}^{\theta^{\star,m}})_{n \geq 1} $ (as defined before with $\theta^{\star} = \sqrt{(s+d)/(r+d)}$)  satisfies the empirical measure theorem. 

In other words, for every  $a,b \   \in \mathbb{R}^d$, 
 $$  \frac{1}{n} \textrm{card}(\{x \in \alpha_{n}^{\theta^{\star, m}}\cap [a,b]  \}) \longrightarrow \frac{1}{C_{f,s}} \int_{[a,b]} f(x)^{\frac{d}{d+s}} dx.$$
\end{prop}
\vskip 0.4cm
\begin{proof}[\textbf{Proof}] \  For all $n \geq 1$,
$$\{x \in \alpha_{n}^{\theta^{\star, m}}\cap [a,b]  \} = \{ x \in \alpha_n \cap [(a-m)/ \theta^{\star} + \mu,(b-m)/ \theta^{\star}]+m  \}. $$
Since  $(\alpha_n)_{n \geq 1}$ is asymptotically  $L^r$-optimal; by  applying the empirical measure theorem to the sequence $(\alpha_n)_{n \geq 1}$, we obtain:
$$ \frac{1}{n}\textrm{card}(\{x \in \alpha_n \cap [(a-m)/ \theta^{\star} + m,(b-m)/ \theta^{\star}+m] \}) \longrightarrow \frac{1}{C_{f,r}} \int_{[(a-m)/ \theta^{\star} + m,(b-m)/ \theta^{\star}+m]} f(x)^{\frac{d}{d+r}} dx.  $$
It remains to verify that $$ \frac{1}{C_{f,r}} \int_{[(a-m)/ \theta^{\star} + m,(b-m)/ \theta^{\star}+m]} f(x)^{\frac{d}{d+r}} dx = \frac{1}{C_{f,s}} \int_{[a,b]} f(x)^{\frac{d}{d+s}} dx.$$ 
Remind that  $$f(x) = \big( (2\pi)^d \textrm{det} \ \Sigma \big)^{-\frac{1}{2}} e^{-\frac{1}{2}(x-m)'\Sigma^{-1}(x-m)}$$  and $ see (\ref{eqmesemp}) \big) $
$$ C_{f,r} = \int_{\mathbb{R}^d} f(x)^{\frac{d}{d+r}} dx.$$
 Hence, for all $r > 0$, 
 $$ C_{f,r} = \big((2 \pi )^d  \textrm{det}\ \Sigma \big)^{\frac{r}{2(r+d)}}\left( \frac{d+r}{d} \right)^{\frac{d}{2}}.$$
 By making the change of variable $x=m +\theta^{\star}(z-m)$, one gets :   
\begin{eqnarray*} 
\frac{1}{C_{f,r}} \int_{[(a-m)/ \theta^{\star} + m,(b-m)/ \theta^{\star}+m]} f(z)^{\frac{d}{d+r}} dz
&  =  & \frac{1}{C_{f,r}} (\theta^{\star})^{-d} \int_{[a,b]} f((x-m)/ \theta^{\star}+m)^{ \frac{d}{d+r}}  dx .
\end{eqnarray*}
It is easy to check that  
$$ \big(f((x-m)/ \theta^{\star}+m\big)^{ \frac{d}{d+r}} = \big(f(x) \big)^{\frac{d}{d+s}} \big( (2\pi)^d \textrm{det}\ \Sigma \big)^{-\frac{1}{2} (\frac{d}{d+r} - \frac{d}{d+s})}  $$
and that  $$ \frac{1}{C_{f,r}} (\theta^{\star})^{-d} \big( (2\pi)^d \textrm{det}\Sigma \big)^{-\frac{1}{2} (\frac{d}{d+r} - \frac{d}{d+s})} =  \big( (2 \pi)^d  \textrm{det}\  \Sigma \big)^{-\frac{s}{2(s+d)}}\left(\frac{d+s}{d} \right)^{-\frac{d}{2}}.$$
The last term is simply equal to  $\frac{1}{C_{f,s}}$. We then deduce that 
$$ \frac{1}{C_{f,r}} \int_{[(a-m)/ \theta^{\star} + m,(b-m)/ \theta^{\star}+m]} f(x)^{\frac{d}{d+r}} dx = \frac{1}{C_{f,s}} \int_{[a,b]} f(x)^{\frac{d}{d+s}} dx. $$
\end{proof}

 We have just built a sequence $(\alpha_n^{\theta^{\star},m})_{n \geq 1}$ verifying the empirical measure theorem. The question we ask know is : is this sequence asymptotically $L^s$-optimal? The next proposition shows that the lower bound in $(\ref{liminf})$ is in fact reached by considering the sequence $(\alpha_n^{\theta^{\star},m})_{n \geq 1}$. 

\begin{prop}
 Let $s>0$ and let  $\theta = \theta^{\star} = \sqrt{(s+d)/(r+d)}$. Then, the constant in the asymptotic lower bound for the $L^s$ error induced by the sequence  $(\alpha_n^{\theta^{\star},m})_{ n \geq 1}$ (see $(\ref{liminf})$) satisfies :
\begin{equation} \label{eqconj1}
Q_{r,s}^{Inf}(P,\theta^{\star}) = Q_s(P).
\end{equation} 
\end{prop}

\begin{proof}[\textbf{Proof}] \ Keep in mind that if $P \sim \mathcal{N}(m;\Sigma)$ then, for all $r>0$, 
$$ \big(Q_r(P) \big)^{1/r} = \big(J_{r,d} \big)^{1/r} \sqrt{ 2 \pi} \left( \frac{d+r}{d} \right)^{\frac{d+r}{2r}} \big(\textrm{det}\ \Sigma \big)^{\frac{1}{2d}}.  $$ 
We have in one hand 
  \begin{eqnarray*}
 \left( \int_{\mathbb{R}^d} f^{\frac{d}{d+r}}(x) d(x)   \right)^{s/d} 
& = &  \left(\big( (2\pi)^d \textrm{det}\ \Sigma \big)^{-\frac{1}{2} \frac{d}{d+r}} \int_{\mathbb{R}^d} e^{-\frac{1}{2} \frac{d}{d+r}(x-m)'\Sigma^{-1}(x-m)} dx \right)^{s/d} \\
& = & \left(\big( (2\pi)^d \textrm{det}\ \Sigma \big)^{\frac{1}{2} \frac{r}{d+r}} \big( \frac{d+r}{d} \big)^{\frac{d}{2}} \right)^{s/d} 
 \end{eqnarray*}
and in the other hand
\begin{eqnarray*}
\int_{\mathbb{R}^d} f_{\theta^{\star}, \mu}(x) f^{- \frac{s}{d+r}}(x) d(x)
& = & \big( (2\pi)^d \textrm{det}\ \Sigma \big)^{-\frac{1}{2}-\frac{s}{d+r}} \int_{\mathbb{R}^d} e^{-\frac{1}{2} \frac{d}{d+r}(x-m)'\Sigma^{-1}(x-m)} dx \\
& = & \big( (2\pi)^d \textrm{det}\ \Sigma \big)^{-\frac{s}{d+r}} \big( \frac{d+r}{d}\big)^{\frac{d}{2}}.
\end{eqnarray*}
Combining these two results yields

\begin{eqnarray*}
Q_{r,s}^{\textrm{Inf}}(P,\theta^{\star})
& = & (\theta^{\star})^{s+d}  J_{s,d} \left(  \int_{\mathbb{R}^d} f^{\frac{d}{d+r}} d\lambda_d   \right)^{s/d} \int_{\mathbb{R}^d} f_{\theta^{\star}, \mu} f^{- \frac{s}{d+r}} d\lambda_d \\ 
& = & J_{s,d} \left( \frac{s+d}{r+d} \right)^{\frac{d+s}{2}}  \big( (2\pi)^d \textrm{det}\ \Sigma \big)^{\frac{s}{2d}} \left( \frac{r+d}{d}\right)^{\frac{d+s}{2}} \\
& = & J_{s,d} \left( \frac{s+d}{d} \right)^{\frac{d+s}{2}}  \big( (2\pi)^d \textrm{det}\ \Sigma \big)^{\frac{s}{2d}} \\
& = & Q_s(P).
\end{eqnarray*}

\end{proof}

After some elementary calculations, it follows from the proposition above and inequalities $(\ref{liminf})$,$(\ref{ineq_princ})$, the corollary below :
\begin{cor} Let $X \sim \mathcal{N}(m;\Sigma)\ $ and $\ \theta^{\star} = \sqrt{(s+d)/(r+d)}$. Then,
\begin{equation}
 Q_s(P)^{1/s} \leq \liminf_{n \rightarrow \infty} n^{1/d} \ \Vert X - \widehat{X}^{\alpha_n^{\theta^{\star},m}} \Vert_s  \leq \limsup_{n \rightarrow \infty} n^{1/d} \ \Vert X - \widehat{X}^{\alpha_n^{\theta^{\star},m}} \Vert_s \leq Q_{r,s}^{\textrm{Sup}}(P,\theta^{\star})^{1/s}
\end{equation}
with 
$$
Q_{r,s}^{\textrm{Sup}}(P,\theta^{\star})^{1/s}=\left \{\begin{array}{ll}
  \big(\frac{s+d}{d} \big)^{\frac{s+d}{2s}} J_{r,d}^{\frac{1}{r}} \ \big((2 \pi)^d \textrm{det }\Sigma \big)^{\frac{1}{2d}}  & \textrm{ if } s<r\\

 \big( \frac{s+d}{d} \big)^{\frac{d}{2}}  \sqrt{\frac{s+d}{r+d}} \ C(b) \ \big((2 \pi)^d \textrm{det }\Sigma \big)^{\frac{1}{2(d+r)}}  & \textrm{ if } s>r.
  \end{array} \right.
$$
\end{cor}
%
\vskip 0.4cm
\begin{rem} \label{rem_gauss}
$(a) \ $  If $s>r$, we cannot  prove the asymptotically $L^s(P)$-optimality of $(\alpha_n^{\theta^{\star},m})_{n\geq 1}$   using $(\ref{limsup2})$ since the constant $C(b)$ is not explicit.

$(b)$ When $s<r$, the corollary above shows that the upper bound  in $(\ref{limsup1})$  does not reach the Zador constant. Then our upper  estimate does not allow us to show that the sequence $(\alpha_n^{\theta^{\star},m})_{ n \geq 1}$ is asymptotically $L^s(P)$-optimal.
 \end{rem}
Moreover, using $H\ddot{o}lder$ inequality (with $p=r/(r-s)$ and $q =r/s $), we have for every $\theta>0$, 
\begin{eqnarray*}
\int_{\mathbb{R}^d}  f_{\theta, \mu}(x) f^{- \frac{s}{d+r}}(x) d\lambda_d(x) & = &  \int_{\mathbb{R}^d} f_{\theta, \mu}(x) f^{-s/r}(x) f^{\frac{sd}{r(d+r)}}(x) d\lambda_d(x) \\
& \leq & \left( \int_{\mathbb{R}^d} f_{\theta,\mu}^{\frac{r}{r-s}}(x) f^{-\frac{s}{r-s}}(x) d\lambda_d(x) \right)^{\frac{r-s}{r}}  \left( \int_{\mathbb{R}^d} f^{\frac{d}{d+r}}(x) d\lambda_d \right)^{\frac{s}{r}} .
\end{eqnarray*}
and (for $\theta= \theta^{\star}$)
\begin{equation} \label{eqconj2}
 \int_{\mathbb{R}^d}  f_{\theta^{\star}, \mu}(x) f^{- \frac{s}{d+r}}(x) d\lambda_d(x) = \left( \int_{\mathbb{R}^d} f_{\theta^{\star},\mu}^{\frac{r}{r-s}}(x) f^{-\frac{s}{r-s}}(x) d\lambda_d(x) \right)^{\frac{r-s}{r}}  \left( \int_{\mathbb{R}^d} f^{\frac{d}{d+r}}(x) d\lambda_d \right)^{\frac{s}{r}}.
\end{equation}
 Hence,  according to $ (\ref{eqconj1})$, one gets  for every $s<r$, 
\begin{equation} \label{eq_notZc} 
(\theta^{\star})^{s+d}  J_{s,d} \left( \int_{\mathbb{R}^d} f_{\theta^{\star},\mu}^{\frac{r}{r-s}}(x) f^{-\frac{s}{r-s}}(x) d\lambda_d(x) \right)^{\frac{r-s}{r}} \Vert f \Vert_{\frac{d}{d+r}}^{s/r}  =Q_s(P).
\end{equation}
Thus, to reach the Zador constant in $(\ref{limsup1})$ we must rather   have  $J_{s,d}$  instead of $J_{r,d}$ (which will be coherent  since  for all $s<r, \ J_{s,d}^{1/s} \leq  J_{r,d}^{1/r}$), that  is,
$$ \limsup_{n \rightarrow \infty} n^{1/d} \ \Vert X - \widehat{X}^{\alpha_n^{\theta,\mu}} \Vert_s \leq  \theta^{s+d}  J_{s,d} \left( \int_{\mathbb{R}^d} f_{\theta,\mu}^{\frac{r}{r-s}}(x) f^{-\frac{s}{r-s}}(x) d\lambda_d(x) \right)^{\frac{r-s}{r}} \Vert f \Vert_{\frac{d}{d+r}}^{s/r}. $$
.

\subsubsection{Numerical experiments}
For numerical example, supppose  that  $d=1$ and  $ r \in \{1,2,4 \}$. Let  $X\sim \mathcal{N}(0,1)$ and, for a fixed $n$, let $ \alpha_{n,r} =\{x_{1,r}, \cdots, x_{n,r} \}$ be the $n$-$L^r$-optimal grid for $X$ (obtained by a Newton-Raphson zero search). For every  $ n \in \{20,50,\cdots,900  \} $  and for   $\ (s,r ) =(1,2) $ and $(4,2)$,  we  make a linear regression of $\alpha_{n,r} $ onto $\alpha_{n,s} $ :

 $$ x_{i,s} \simeq \hat{a}_{sr} x_{i,r} + \hat{b}_{sr},\quad i=1,\cdots,n .$$
 Table  $\ref{tab1}$ provides the  regression coefficients  we obtain for different values of $n$. We note that when $ n $ increases, the coefficients $\hat{a}_{sr}$ tend to the value $\sqrt{(s+1)/(r+1)} = \theta^{\star}$ whereas the coefficients $\hat{b}_{sr}$ almost vanish. For example, for  $n=900$ and for $ (r,s)=(2,1)$ (resp. $(2,4)$) we get $\hat{a}_{sr}= 0.8170251$ (resp. $ 1.2900417 $). The expected  values are $\sqrt{2/3} = 0.8164966 $ (resp. $\sqrt{5/3} = 1.2909944 $). The absolute errors are then ${5.285}\times 10^{-4} $ $\big(\textrm{resp.}\  9.527 \times 10^{-4} \big) $. We remark that the  error mainly comes from the  tail  of the distribution.
\vskip 0.4cm
\begin{table}[htbp]
 \begin{center}
  \begin{tabular}{|*{8}{c|}}
  \hline
 $ n $ & $\hat{a}_{12}$ & $ \hat{b}_{12} $ & $\epsilon $ & & $\hat{a}_{42}$ & $ \hat{b}_{42} $ & $\epsilon $  \\
  \hline  
20 & 0.8250096 & 1.826E-14 & 0.0003025 & &  1.2761027 &  - 3.650E-12 &  0.0008607 \\
  \hline
50 & 0.8211387  & - 1.021E-13 & 0.0006870 & &  1.2828110   & 3.733E-10   &  0.0020110 \\
  \hline
100 &  0.8193424 & 8.693E-14 & 0.0009909 & & 1.2859567  & 4.059E-09 & 0.0029445\\
  \hline
300 &  0.8177506  & - 1.045E-11 & 0.0013601 & & 1.2887640   & 0.0000004   &  0.0041021\\
  \hline
700 &  0.8171428  & - 7.219E-11  & 0.0015111 &  & 1.2898393   &  - 0.0000089   & 0.0048006  \\
  \hline
800 &  0.8170775   & - 6.725E-11 &  0.0015247 &   &1.2900041   &  0.0000216    &  0.0040577 \\
  \hline
900 &  0.8170251  & 4.564E-11 & 0.0015346 &  & 1.2900417  &  - 0.0000141   & 0.0048182  \\
 \hline
\end{tabular}
\caption{\small{Regression coefficients for the Gaussian.  \label{tab1}}}
\end{center}
\end{table}

The  previous numerical results, in addition to Equation $(\ref{eqconj1})$, strongly suggest  that the sequence $(\alpha_n^{\theta^{\star},m})_{n \geq 1}$ is  in fact  asymptotically $L^s(P)$-optimal. This leads to the   following conjecture. 
\begin{conj}  Let $P \sim \mathcal{N}(m;\Sigma)$ and let  $(\alpha_n)_{n \geq 1}$ be an $L^r(P)$-optimal sequence of quantizers.  Then, for every   $s>0$, the sequence  $(\alpha_n^{\theta^{\star},m})_{ n \geq 1}$ (with $\ \theta^{\star} = \sqrt{(s+d)/(r+d)}$)  is asymptotically  $L^s(P)$-optimal.
\end{conj}

\subsection{Exponential distribution } 
\subsubsection{Optimal dilatation and contraction}  
\begin{prop} \label{prop_exp}
 
Let $r,s >0$ and  $X$ be an exponentially distributed random variable with rate parameter $\lambda >0$. Set $\mu = 0.$ 
\begin{enumerate}
\item[\bf{(a)}] If $s \in \big(r,r+1 \big) \cup \big(r+1,+\infty \big) $, the sequence  $( \alpha_n^{\theta,0})_{n \geq 1}$ is $L^s$-rate-optimal iff   $ \ \theta \in \big(s/(r+1),+\infty \big)$ and $$ \theta^{\star} = (s+1)/(r+1) $$ is the unique solution of $(\ref{eqmin})$ on the set  $\big(s/(r+1),+\infty \big).$
\item[\bf{(b)}] If $s \in (0,r) $, the sequence  $( \alpha_n^{\theta,0})_{n \geq 1}$ is $L^s$-rate-optimal for all $ \ \theta \in \big(s/r,+\infty \big)$   and  $$ \theta^{\star} = (s+1)/(r+1)$$ is the unique  solution of $(\ref{eqmin})$ on $\big(s/r,+\infty \big).$
\end{enumerate}
\end{prop}
\vskip 0.4cm
\begin{proof}[\textbf{Proof}]

 $\mbox{\bf{(a)}} \ $ Let  $s \in (r,r+1)$. For all $\theta>0, \mu \in \mathbb{R}^d$, the couple $(\theta,\mu)$ is $P$-admissible and the  function $f$ is decreasing on  $(0,+\infty)$.  For $\theta>s/(r+1)$, Assumption $(\ref{asscor1})$ holds true for every  $c \in \big(1, \theta(1+r)/s \big)$. Moreover,  Hypothesis $\textbf{(H1)}$ is clearly satisfied. Consequently, if follows from Corollary  $\ref{cor3.1},  (a)$ that  Assumption $(\ref{foncmax})$ holds true.

If $s>r+1$, Assumption  $(\ref{foncmax})$ still holds since the  additionnal assumptions $\textbf{(H3)}$ and $f^{-\frac{s}{r+1}} \in L_{loc}^{1}(P)$ required to apply  the corollary $\ref{cor3.1},  (b)$ are satisfied.

In the other hand, one has
\begin{equation*}
\int_{\mathbb{R}} f(\theta x) f(x)^{-s/(r+1)}dx = C \int_{0}^{+\infty}e^{-\lambda (\theta - s/(r+1)) x} dx < + \infty \Longleftrightarrow  \theta > s/(r+1).
\end{equation*}

Now, let us solve the problem $(\ref{eqmin})$  For all $ \ \theta > s/(r+1)$, 
\begin{eqnarray*}
\theta^{s+1} \int_{\mathbb{R}} f(\theta x) f(x)^{-\frac{s}{r+1}}dx 
& = & C \ \theta^{s+1} \int_{0}^{+\infty}e^{-\lambda (\theta - \frac{s}{r+1}) x}dx \\
& = & C \ \theta^{s+1} \left(\theta -\frac{s}{r+1} \right)^{-1} . 
\end{eqnarray*}
Let $$ h(\theta) = \theta^{s+1} \left(\theta -\frac{s}{r+1} \right)^{-1}. $$
Then  $$ h'(\theta) = s \theta^{s} \left(\theta -\frac{s}{r+1} \right)^{-2} \left( \theta - \frac{s+1}{r+1} \right). $$
Hence, $h$ reaches its unique  minimun  on  $ \big(s/(r+1),+\infty)$ at  $ \theta^{\star} = (s+1)/(r+1).$ 

$\mbox{\bf{(b)}}$  Let $s<r$.  Then
$$ \int_{\mathbb{R}} f^{\frac{r}{r-s}} (\theta x) f^{-\frac{s}{r-s}}(x) dx = C \int_{\mathbb{R}_{+}}  e^{-x \frac{\lambda}{r-s}( r \theta - s)} dx.  $$
Then, for all  $\theta > s/r, \  \int_{\mathbb{R}} f^{\frac{r}{r-s}} (\theta x) f^{-\frac{s}{r-s}}(x) dx < +\infty .$  This gives the first assertion. \\
For all  $ \theta > s/r $, then
\begin{eqnarray*}
\theta ^{s+1} \left( \int_{\mathbb{R}} f_{\theta,\mu}^{\frac{r}{r-s}} (x) f^{-\frac{s}{r-s}}(x) dx \right)^{1-\frac{s}{r}} 
& = & C \ \theta^{s+1} \left(\int_{\mathbb{R}_{+}}  e^{-x \frac{\lambda}{r-s}( r \theta - s)} dx \right)^{\frac{r-s}{r}} \\
& = &  C \ \theta^{s+1} \left( r \theta - s \right)^{\frac{s-r}{r}}.
\end{eqnarray*}
We easily check that the function  $h(\theta) = \theta^{s+1} \left( r \theta - s \right)^{\frac{s-r}{r}}$  reaches its minimum on  $\big(s/r,+\infty)$ at the unique point $ \theta^{\star} = (s+1)/(r+1).$

\end{proof}

\begin{rem} 
Let  $X \sim \mathcal{E}(\lambda).$  If $s<r,\textrm{ then } \ \theta^{\star}  = (s+1)/(r+1)< 1$. Hence, the sequence $\ (\alpha_n^{\theta^{\star},0})_{n \geq 1}$  is a contraction of $(\alpha_n)_{n\geq 1}$ with scaling number $\theta^{\star}$. In the other hand, if $s > r, \textrm{ then } \ \theta^{\star} > 1$ and then  $\ (\alpha_n^{\theta^{\star},0})_{n \geq 1}$ is a dilatation of  $(\alpha_n)_{n \geq 1}$ with scaling number $\theta^{\star}$. Note that  $\theta^{\star}$ does not depend on the parameter $\lambda$ of the exponential  distribution.
\end{rem}

  One shows below that the sequence  $ (\alpha_n^{\theta^{\star},0})_{n \geq 1}$, with  $  \theta^{\star}   =   (1+s)/(1+r)  $, satisfies the empirical measure theorem.
\begin{prop} \label{prop_empthm2}
Let $r,s>0$ and let  $X$ be  an exponentially distributed random variable with rate parameter $\lambda>0$. Assume  $(\alpha_{n})_{n \geq 1} $  is an asymptotically  $L^r$-optimal sequence of quantizers for  $X$  and  let  $ (\alpha_{n}^{\theta^{\star},0})_{n \geq 1}$  be defined as before, with  $\ \theta^{\star} = (s+1)/(r+1)$.  Then, the sequence  $ (\alpha_{n}^{\theta^{\star},0}) $  satisfies  the empirical measure theorem.
\end{prop}
\vskip 0.4cm
 \begin{proof}[\textbf{Proof}]  Since $ (\alpha_n^{\theta^{\star},0})_{n\geq 1} = ( \theta^{\star} \alpha_n)_{n \geq 1}$, It amounts to show that $$ \frac{\textrm{card}(\alpha_n \cap [a/ \theta^{\star},b/ \theta^{\star}])}{n} \  \longrightarrow \ \frac{1}{C_{f,s}}  \int_a^b f(x)^{\frac{1}{1+s}} dx $$ 
i.e that for all  $a,b \in \mathbb{R}_{+}$,    $$   \frac{1}{C_{f,r}} \frac{1}{\theta^{\star}} \int_a^b f(x/ \theta^{\star})^{ \frac{1}{1+r}}  dx = \frac{1}{C_{f,s}} \int_a^b f(x)^{\frac{1}{1+s}} dx.  $$
  Elementary computations show that  $ \forall \ r > 0 $, 
$$ C_{f,r} = \lambda^{-\frac{r}{1+r}} (1+r) .$$
 so that
\begin{eqnarray*}
 \frac{1}{C_{f,r}} \frac{1}{\theta^{\star}} \int_a^b f(x/ \theta^{\star})^{ \frac{1}{1+r}}  dx 
 & = & \frac{1}{C_{f,r}}  \frac{1+r}{1+s}   \int_a^b \left( \lambda e^{-x \lambda \frac{1+r}{1+s}} \right)^{ \frac{1}{1+r} }  dx\\
 & = & \frac{1}{C_{f,r}} \frac{1+r}{1+s} \  \lambda^{\frac{1}{1+r} - \frac{1}{1+s}} \int_a^b \left( \lambda e^{-\lambda x } \right)^{ \frac{1}{1+s} }  dx \\
 & = & \frac{1}{C_{f,s}}  \int_a^b f(x)^{\frac{1}{1+s}} dx. 
\end{eqnarray*}
\end{proof}

Is the sequence $(\alpha_n^{\theta^{\star},0})_{n \geq 1}$  asymptotically $L^s$-optimal?  The remark $\ref{rem_gauss}$  is also valid for the exponential distribution. Our upper bounds in $(\ref{limsup1})$ and $(\ref{limsup2})$ do not allow us to show that  $(\theta^{\star} \alpha_n)$ is asymptotically $L^s$-optimal because of the corollary below. But the numerical results  strongly suggest that it is.
 \vskip 0.4cm
\begin{cor} Let $X \sim \mathscr{E}(\lambda)\ $ and $ \ \theta^{\star} = (s+1)/(r+1)$. Then,
\begin{equation}
 Q_s(P)^{1/s} \leq \liminf_{n \rightarrow \infty} \ n^{1/d}  \Vert X - \widehat{X}^{\alpha_n^{\theta^{\star},0}} \Vert_s  \leq \limsup_{n \rightarrow \infty}  n^{1/d} \ \Vert X - \widehat{X}^{\alpha_n^{\theta^{\star},0}} \Vert_s \leq Q_{r,s}^{\textrm{Sup}}(P,\theta^{\star})^{1/s}
\end{equation}
with 
$$
Q_{r,s}^{\textrm{Sup}}(P,\theta^{\star})^{1/s}=\left \{\begin{array}{ll}
  \frac{1}{2\lambda}(s+1)^{1+ 1/s} (r+1)^{- 1/r}  & \textrm{ if } s<r\\

(s+1)^{1+1/s} \big((r+1) \lambda^{\frac{1}{1+r}} \big)^{-1} C(b)^{1/s}  & \textrm{ if } s>r.
  \end{array} \right.
$$
\end{cor}
\vskip 0.4cm
\begin{proof}[\textbf{Proof}]
We easily prove, like in proposition \ref{rem_gauss}, that $Q_{r,s}^{\textrm{Inf}}(P,\theta^{\star}) = Q_s(P)$.  The corollary follows then from $(\ref{liminf})$ and  $(\ref{ineq_princ})$ \big(keep in mind that for all $\ r>0, \quad  J_{r,1} = \frac{1}{(r+1) 2^r} \ $ \big).
\end{proof}
\subsubsection{Numerical experiments}

 For numerical examples, Table  $\ref{tab2}$ gives the regression coefficients we obtain by regressing the $L^2$ grids onto the grids we get with the $L^1$ and $L^4$ norms, for different values of $n$. The notations are the same as in  the previous example. We note that for large enough $ n$, the coefficients $\hat{a}_{sr}$ tend to $(s+1)/(r+1) = \theta^{\star}$. For example, if $n=900$, we get $\hat{a}_{12} = 0.6676880; \  \hat{a}_{42} = 1.6640023  $  whereas the expected values are respectively  $2/3 = 0.66666667 $ and $5/3=1.6666667 $. The absolute errors are in the  order of  $10^{-3}$. Like the Gaussian case, we remark that the error of the estimation results mainly from the  tail  of the exponential distribution.
\vskip 0.4cm
\begin{table}[htbp]
 \begin{center}
  \begin{tabular}{|*{8}{c|}}
  \hline
 $ n $ & $\hat{a}_{12}$ & $ \hat{b}_{12} $ & $\epsilon $ && $\hat{a}_{42}$ & $ \hat{b}_{42} $ & $\epsilon $    \\
  \hline  
20 & 0.6765013 &  - 0.0104881 & 0.0019489 && 1.6396807 & 0.0288348 & 3.081E-33\\
  \hline
50 &  0.6726145  & - 0.0082123 & 0.0045310 && 1.6502245 & 0.0225246 & 1.149E-28 \\
  \hline
100 &  0.6706176 & - 0.0062439 & 0.0070734 && 1.6556979  & 0.0172020 & 1.573E-27 \\
  \hline
300 &   0.6686428 & - 0.0036234 & 0.0114628 && 1.6611520 &  0.0100523 & 1.508E-27\\
  \hline
700 &  0.6677864  &  - 0.0022222  & 0.0146186 && 1.6635261  & 0.0061356 & 1.222E-25 \\
  \hline
800 &   0.6676880  &  - 0.0020482   &   0.0150735 && 1.6638043 & 0.0057199 & 2.020E-26 \\
  \hline
900 &   0.6676079 &  - 0.0019043  &  0.0154634 && 1.6640023 & 0.0053173  & 9.683E-25 \\
 \hline
\end{tabular}
\caption{\small{ Regression coefficients  for exponential distribution.} \label{tab2}}
\end{center}
 
\end{table}
\begin{conj}  Let  $X$ be an exponentially distributed random variable  with  rate parameter   $\lambda$ and let   $(\alpha_n)_{n \geq 1}$ be an $L^r$-optimal sequence of quantizers for  $X$. Then for  $s>0$ and $\ \theta^{\star} = (s+1)/(r+1)$  the sequence  $(\alpha_n^{\theta^{\star},0})_{ n \geq 1}$  is asymptotically   $L^s$-optimal.
\end{conj}

These  exampless could suggest that a contraction (or a dilatation) parameter $\theta^{\star}$, solution of the minimisation problem $(\ref{eqmin})$, always leads to a  sequence of  quantizers satisfying the empirical measure theorem. The following example shows that  this can fail.

\subsection{Gamma distribution}  
\subsubsection{Optimal dilatation and contraction}
\begin{prop} \label{prop_gaus}
Let  $r,s >0$  and  let  $X$  be a  Gamma distribution  with parameters  $a \textrm{  and  } \lambda : X \sim \ \Gamma(a,\lambda), \ a>0, \ \lambda>0$. 
\begin{enumerate}
\item[\bf{(a)}] if  $s \in (r,r+1)$, the sequence   $(\alpha_n^{\theta,0})_{n \geq 1}$ is  $L^s$-rate-optimal  iff   $ \ \theta \in \big(s/(r+1),+\infty \big)$  and for all  $a>0$, $$ \theta^{\star} = (s+a)/(r+a) $$  is the unique solution of  $(\ref{eqmin})$ on the set  $\big(s/(r+1),+\infty \big).$
\item[\bf{(b)}] if $ s>r+1 $ and if  $a \in \big(0, \frac{s+r+1}{s} \big)$, the sequence  $( \alpha_n^{\theta,0})_{n \geq 1}$ is  $L^s$-rate-optimal for every   $ \ \theta \in \big(s/(r+1),+\infty \big)$ and  $$ \theta^{\star} = (s+a)/(r+a)   $$ is the unique solution  of  $(\ref{eqmin})$  on the set  $\big(s/(r+1),+\infty \big)$ (Note that the assumptions  imply $a \in (0,2)$). 
\item[\bf{(c)}] if  $s<r$,  the sequence   $(\alpha_n^{\theta,0})_{n \geq 1}$ is  $L^s$-rate-optimal  for every  $ \ \theta \in \big(s/r,+\infty \big)$  and for all  $a>0$, $$ \theta^{\star} = (s+1)/(r+1)$$  is the  unique solution of  $(\ref{eqmin})$ on the set  $\big(s/r,+\infty \big).$
\end{enumerate}
\end{prop}
\vskip 0.4cm
\begin{proof}[\textbf{Proof}]  We set   $\mu =0$.
 Keep in mind that  the density function is written  $$f(x)=\frac{\lambda^a}{\Gamma(a)}x^{a-1}e^{-\lambda x} \mbox{\bf{1}}_{\{x>0\}}, \textrm{ with }  \Gamma(a) = \int_0^{+\infty} x^{a-1}e^{-x} dx. $$

$\mbox{\bf{(a)}} \textrm{ and } \mbox{\bf{(b)}} $. Let  $s \in (r,r+1)$ and  set  $R_0 = \max(0, (a-1)/ \lambda)$. The function  $f$  is decreasing  on  $(R_0,+\infty)$  and for every  $\theta>0, \mu$, the couple $(\theta,\mu)$ is $P$-admissible.   For $\theta> s/(r+1)$, Assumption $(\ref{asscor1})$    holds true  for every  $c \in \big(1, \theta(1+r)/s \big)$. Moreover, Hypothesis $\textbf{(H1)}$ clearly holds.  Consequently, it follows from Corollary $\ref{cor3.1}, (a)$ that  Assumption   $(\ref{foncmax})$ of  Theoreme $\ref{thm3}$ holds true.

When $s>r+1$, the additionnal hypothesis $f^{-\frac{s}{r+1}} \in L_{loc}^{1} (P)$ holds for $a < \frac{r+1}{s} +1$. Note that if $P=f \cdot \lambda_d$ then $ \lambda_d(\textrm{supp}(P) \cap \{f=0\}) = 0$ implies that $\lambda_d(\cdot \cap \textrm{supp}(P)) \ll P.$
 It follows that $\textbf{(H3)}$ holds. In this case Assumption   $(\ref{foncmax})$ of  Theoreme $\ref{thm3}$ holds true.

For all  $\theta>0$,
\begin{equation*} 
\int_{\mathbb{R}} f(\theta x) f(x)^{-\frac{s}{1+r}}dx = \left( \frac{\lambda^a}{\Gamma(a)} \right)^{1-s/(r+1)} \int_0^{+\infty} x^{(a-1)(1-\frac{s}{r+1})} e^{-(\theta - \frac{s}{r+1})\lambda x} dx
\end{equation*}
and then  $$\int_{\mathbb{R}} f(\theta x) f(x)^{-\frac{s}{r+1}}dx < +\infty \quad \textrm{ iff }  \quad \theta > s/(r+1) \ \textrm{ and } \  a(r+1-s)+s>0 .$$ 

Let  $ \theta > s/(r+1) $.  Then 
\begin{eqnarray*}
 \theta^{s+1} \int_{\mathbb{R}} f(\theta x) f(x)^{-\frac{s}{1+r}}dx 
& = & \left( \frac{\lambda^a}{\Gamma(a)} \right)^{1-s/(r+1)}  \ \theta^{s+1} \theta^{a-1}  \int_0^{+\infty} x^{(a-1)(1-\frac{s}{1+r})} e^{-(\theta - \frac{s}{1+r})\lambda x} dx  \\
& = & C \ \theta^{\gamma} \left(\theta-\frac{s}{1+r}\right)^{-\beta} .
\end{eqnarray*}
with  $$ \gamma =s+a \ \textrm{  and   } \  \beta=(a-1)(1- s/(r+1)) +1.  $$
We define on  $\mathbb{R}_{+}^{\star}$ the function  $h$  by  $$h(\theta) = \theta^{\gamma} \left(\theta-\frac{s}{1+r} \right)^{-\beta}. $$
The function  $h$  is differentiable  for all $\theta > s/(1+r)$ and  $$h'(\theta) = \theta^{\gamma -1} \left(\theta-\frac{s}{1+r} \right)^{-\beta-1}\left((\gamma - \beta)\theta - \frac{s \gamma}{1+r}\right).$$
Hence, the minimum  of $h$ is then unique on $\big(s/(r+1),+\infty \big)$ and is reached at   $\theta^{\star}.$ 

 Notice that the condition required for $f^{-\frac{s}{r+1}} $ to be in $ L_{loc}^{1} (P)$  is  $a< \frac{r+1}{s} +1$ and for every  $s>r+1$  one has  $1+ \frac{r+1}{s} < \frac{s}{s-(r+1)}$.  Combined to the condition  $a(r+1-s)>0$ yields the condition for  $a$.  \\
 $\mbox{\bf{(c)}} \ $  Let  $s<r$. Then
\begin{equation*}
\int_{\mathbb{R}} f^{\frac{r}{r-s}} (\theta x) f^{-\frac{s}{r-s}}(x) dx = \frac{\lambda^a}{\Gamma(a)} \int_0^{+\infty} x^{a-1} e^{- \frac{\lambda x}{r-s}(r \theta -s)} dx.
\end{equation*}
Therefore  $ \int_{\mathbb{R}} f^{\frac{r}{r-s}} (\theta x) f^{-\frac{s}{r-s}}(x) dx < +\infty \ $ iff $\ \theta > s/r.$ 

Let  $\theta > s/r.$  Then

\begin{eqnarray*}
\theta ^{1+s} \left( \int_{\mathbb{R}} f_{\theta,\mu}^{\frac{r}{r-s}} (x) f^{-\frac{s}{r-s}}(x) dx \right)^{1-\frac{s}{r}}
& = &  C \  \theta^{s+a} \left( \int_0^{+\infty} x^{a-1} e^{- \frac{\lambda x}{r-s}(r \theta -s)} dx  \right)^{\frac{r-s}{r}} \\
& = &  C \ \theta^{s+a} \big( r \theta -s \big)^{a \frac{s-r}{r}}.
\end{eqnarray*}
Considering the  function  $h$ defined by  $ h(\theta) = \theta^{s+a} \left( r\theta -s \right)^{a\frac{s-r}{r}} $ we show that  $h$  reached its minimum on  $\big(s/r,+\infty \big)$  at the unique point  $\theta^{\star} = (s+a)/(r+a).$

\end{proof}

\begin{rem} Let   $X \sim \Gamma(a,\lambda)$.  If  $s<r,\textrm{ then } \ \theta^{\star}  = (s+a)/(r+a)< 1$. Then the sequence  $\ (\alpha_n^{\theta^{\star},0})_{n \geq 1}$ is a contraction  of  $(\alpha_n)_{n \geq 1}$ with  scaling number  $\theta^{\star}$. On the other hand, if   $s > r , \textrm{ then  } \ \theta^{\star} > 1$ and the sequence  $\ (\alpha_n^{\theta^{\star},0})_{n \geq 1}$ is a dilatation of  $(\alpha_n)_{n \geq 1}$  with scaling number  $\theta^{\star}$. Moreover there is no constraint on the  parameter  $a$  as long as  $ s < r+1.$ In this case when we set  $ a=1 $ (exponential  distribution with  parameter  $\lambda$)  we retrieve  the result of the exponential distribution. Note that   $\theta^{\star}$  does not depend  on the parameter $\lambda$. That is expected  since    $\Gamma(1,\lambda) = \mathscr{E}(\lambda)$ and, in the exponential case we know that  the scaling number does not depend on  $\lambda$.
\end{rem}

 Let $\theta^{\star} = (s+a)/(r+a)$ and consider  now  the sequence  $(\alpha_n^{\theta^{\star},0})_{n \geq 1}$  defined  as previously. Does this sequence verify the empirical measure theorem?   If  $a=1$  we  boil down to  the exponential distribution.   On the other hand, when   $a \not= 1$, one  shows   below that  there exists  $a>1,\ s>0$ and $r>0$  such that  the sequence  $(\alpha_n^{\theta^{\star},0})_{n \geq 1}$  does not verify  the empirical measure theorem.

 Suppose  that   $(\alpha_n^{\theta^{\star},0})_{n \geq 1}$  satisfies  the  empirical measure theorem. Then we must  have, for all  $u \ \in \mathbb{R}_{+}$, 
\begin{equation} \label{eqempir}
  \frac{1}{C_{f,r}} \frac{1}{\theta^{\star}} \int_0^u f(x/ \theta^{\star})^{ \frac{1}{1+r}}  dx = \frac{1}{C_{f,s}} \int_0^u f(x)^{\frac{1}{1+s}} dx.   
 \end{equation}
 with   $  \ f(x)=\frac{\lambda^a}{\Gamma(a)}x^{a-1}e^{-\lambda x} \mbox{\bf{1}}_{\{x>0\}} \ $ and  $ \ C_{f,r} = \int f(x)^{\frac{1}{1+r}} dx  \ \textrm{ for all } \ r>0.$ 

Moreover, let  $r>0$. Then,
 \begin{eqnarray*} 
 C_{f,r}  
 & = &   \lambda^{\frac{a}{1+r}} \Gamma(a)^{-\frac{1}{1+r}}  \int_0^{+\infty} x^{(a-1)/(r+1)} e^{-\frac{\lambda}{1+r}x} dx  \\ & = &
 \lambda^{\frac{a}{1+r}} \Gamma(a)^{-\frac{1}{1+r}}  \int_0^{+\infty} x^{(r+a)/(r+1)-1} e^{-\frac{\lambda}{1+r}x} dx    \\ & = &
 \lambda^{\frac{a}{1+r}} \Gamma(a)^{-\frac{1}{1+r}}  \Gamma \left(\frac{r+a}{r+1} \right)   \lambda^{-\frac{r+a}{r+1}}   \big( r+1 \big)^{\frac{r+a}{r+1}}   \\ & = &
 \Gamma \left(\frac{r+a}{r+1} \right)  \Gamma(a) ^{-\frac{1}{1+r}} \lambda^{-\frac{r}{r+1}}  \big( r+1 \big)^{\frac{r+a}{r+1}}. 
 \end{eqnarray*}
Equation  $(\ref{eqempir})$ is written down  for all  $u \ \in \mathbb{R}_{+}$,
$$ C(r) \left( \frac{r+a}{s+a} \right)^{\frac{r+a}{r+1}} \int_0^u  x^{\frac{a-1}{r+1}} e^{-\frac{\lambda(r+a)}{(r+1)(s+a)} x} dx = C(s)  \int_0^u  x^{\frac{a-1}{s+1}} e^{-\frac{\lambda}{s+1} x} dx$$
with  $  \ C(r) =  \Gamma \left(\frac{r+a}{r+1} \right)^{-1}  \lambda^{\frac{r+a}{r+1}}  \big( r+1 \big)^{-\frac{r+a}{r+1}},  \qquad  \forall \  r> 0.$ 

Let  $m \in \mathbb{N}$ and   $\alpha >0$.  We show by induction that, for    $u >0$, 
$$ \int_0^u  x^n e^{-\alpha x} dx = - \left( \frac{1}{\alpha} u^n + \frac{n}{\alpha^2} u^{n-1} + \frac{n(n-1)}{\alpha^3} u^{n-2} + \cdots + \frac{n!}{\alpha^n} u + \frac{n!}{\alpha^{n+1}} \right) e^{- \alpha u} + \frac{n!}{\alpha^{n+1}}. $$
Consider   $a>1$  such that  $ \frac{a-1}{r+1} $  and  $\frac{a-1}{s+1}$  are integers. Set  $n=\frac{a-1}{r+1}, \ m = \frac{a-1}{s+1}, \ \alpha= \frac{\lambda(r+a)}{(r+1)(s+a)} \textrm{ and  } \beta= \frac{\lambda}{s+1}.$
 Then  Equation  $(\ref{eqempir})$ is   finally  written down
\setlength\arraycolsep{0pt}
\begin{eqnarray*} 
C(r)  &  & \left( \frac{r+a}{s+a} \right)^{\frac{r+a}{r+1}}   \bigg[  \left( \frac{1}{\alpha} u^n + \frac{n}{\alpha^2} u^{n-1} + \frac{n(n-1)}{\alpha^3} u^{n-2} + \cdots + \frac{n!}{\alpha^n} u + \frac{n!}{\alpha^{n+1}} \right) e^{- \alpha u} - \frac{n!}{\alpha^{n+1}} \bigg]\\
& = & C(s) \bigg[\left( \frac{1}{\beta} u^m + \frac{m}{\beta^2} u^{m-1} + \frac{m(m-1)}{\beta^3} u^{m-2} + \cdots + \frac{m!}{\beta^m} u + \frac{m!}{\beta^{m+1}} \right) e^{- \beta u } - \frac{m!}{\beta^{m+1}} \bigg].
\end{eqnarray*}
Set   $a=7, \ s=1, \ r=2, \ \lambda=1 \textrm{ and  } u=1$.  Then  $n=2, m=3,\alpha=3/8, \beta = 1/2$  and this lead, after some calculations to   : $$\frac{185}{128} e^{-3/8}-\frac{79}{48} e^{-1/2} =  -\frac{511}{512};$$  which is clearly not satisfied.  We  then deduce  that for   $(a,r,s)=(7,2,1)$,  the sequence  $(\alpha_n^{\theta^{\star},0})_{n \geq 1} $  does not  satisfy  the empirical measure theorem.  Hence, we have  constructed  an $L^s(P)$-rate-optimal  sequence which does not  satisfy   the  empirical  mesure  theorem.


\end{document}